\documentclass[11pt]{amsart}

\usepackage{graphicx}
\usepackage{amssymb}
\usepackage{latexsym}
\usepackage{cite}

\usepackage[nohug,PostScript=dvips,size=.35in]{diagrams}

\allowdisplaybreaks[1]  

\newcommand{\romanletters}
{\renewcommand{\theenumi}{\roman{enumi}}
\renewcommand{\labelenumi}{\textup{(}\theenumi\textup{)}}}

\newcommand{\letters}
{\renewcommand{\theenumi}{\alph{enumi}}
\renewcommand{\labelenumi}{\textup{(}\theenumi\textup{)}}}

\newtheorem{theorem}{Theorem}
\newtheorem{proposition}[theorem]{Proposition}
\newtheorem{corollary}[theorem]{Corollary}
\newtheorem{lemma}[theorem]{Lemma}
\newtheorem{definition}[theorem]{Definition}

\newcommand{\<}{\langle}
\renewcommand{\>}{\rangle}

\renewcommand{\)}{\textup{)}}
\newcommand{\p}{\varphi}
\renewcommand{\b}{\Lambda}
\renewcommand{\phi}{\varphi}
\newcommand{\Mbar}{\smash[t]{\overline{M}}}
\newcommand{\Bbar}{\smash[t]{\overline{B}}}
\newcommand{\e}{\varepsilon}
\newcommand{\R}{\mathbb R}
\newcommand{\U}{\mathcal{U}}

\newcommand{\h}{\mathbb{H}}
\newcommand{\X}{\Omega}
\newcommand{\Hbar}{\smash[t]{\overline{\h}}}
\newcommand{\linP}{\mathcal{P}}

\newcommand{\Lie}{\mathcal{L}}
\newcommand{\vlap}{L}
\newcommand{\n}{\mathcal{N}}
\newcommand{\A}{\mathcal{A}}
\newcommand{\Q}{\mathcal{Q}}

\newcommand{\ghyp}{{\breve g}}
\newcommand{\rhohyp}{\breve \rho}
\newcommand{\lhyp}{\vlap_{\ghyp}}

\DeclareMathOperator{\Tr}{Tr}

\DeclareMathOperator{\supp}{supp}
\DeclareMathOperator{\divergence}{div}
\DeclareMathOperator{\grad}{grad}

\def\crn#1#2{{\vcenter{\vbox{
  \hbox{\kern#2pt \vrule width.#2pt height#1pt
     }
    \hrule height.#2pt}}}}
\newcommand{\intprod}{\mathchoice\crn54\crn54\crn{3.75}3\crn{2.5}2}
\newcommand{\into}{\mathbin{\intprod}}

\newcounter{mnotecount}[section]
\let\oldmarginpar\marginpar
\renewcommand\marginpar[1]{\-\oldmarginpar[\raggedleft\footnotesize #1]%
{\raggedright\footnotesize #1}}

\begin{document}

\title[Asymptotic gluing]{Asymptotic gluing of asymptotically hyperbolic solutions to the Einstein constraint equations}

\author{James Isenberg, John M. Lee, Iva Stavrov Allen}
\date{\today}

\keywords{constraint equations; asymptotically hyperbolic; gluing.}
\subjclass[2000]{Primary  83C05; Secondary  83C30,  53C21}
\thanks{Research supported in part by NSF grant DMS-0406060 at Washington and NSF grant PHY-0652903 at Oregon.}


\begin{abstract}
We show that asymptotically hyperbolic solutions of the Einstein constraint equations with constant mean curvature can be glued in such a way that their asymptotic regions are connected.
\end{abstract}

\maketitle

\section{Introduction}\label{intro}

One of the most useful ways to produce new solutions of the Einstein constraint equations is via gluing techniques.
The standard gluing construction is the following:
We presume that $(M, g, K)$ is an Einstein initial data set, with $M$ a smooth $n$-dimensional manifold, $g$ a Riemannian metric on $M$, and $K$ a symmetric tensor field on $M$. We further assume that this set of data satisfies the (vacuum) Einstein constraint equations
\begin{align}
&\label{1steq} \mathrm{div}_{g} K- \nabla \Tr_{g}K=0,\\
&\label{2ndeq}R(g)-|K|^2_{g}+(\Tr_{g} K)^2=0,
\end{align}
which are the necessary and sufficient conditions for $(M, g, K)$ to generate a spacetime solution of the (vacuum) Einstein gravitational field equations via the Cauchy problem
\cite{ivon}.
(Here $\mathrm{div}_g$ is the divergence operator, $\Tr_g$ is the trace operator, $R(g)$ is the scalar curvature, and $|\cdot|_g$ is the tensor
norm, all corresponding to the metric $g$.)

Choosing a pair of points $p_1, p_2 \in M$, one shows that
there is
a family of new solutions
$({M_\e}, {g_\e}, {K_\e})$ 
of the constraint equations
in which
(i) ${M_\e}$ is obtained from $M$ by (connected sum) surgery joining $p_1$ and $p_2$
and (ii) outside of a neighborhood of the connected sum bridge in ${M_\e}$, the  data $({g_\e}, {K_\e})$
can be made as close as desired to $(g,K)$ (in a sense to be made precise later)
by taking $\e$ sufficiently small.
We note that this gluing construction allows for the possibility that the manifold $M$  consists of two disconnected components; then if $p_1$ is chosen to lie in one of the components and $p_2$ in the other, the new glued solution effectively connects two disconnected solutions of the constraint equations.

The mathematics and the utility of the gluing of solutions of the Einstein constraint equations are discussed in a series of papers \cite{IMP1, IMP2, CIP, IMaxP}, which show that gluing can be carried out for a wide variety of initial data sets: they can be compact, asymptotically Euclidean (``AE"), or asymptotically hyperbolic (``AH"), and they can be vacuum solutions or non-vacuum solutions with various coupled matter fields. This past work shows that in some cases the gluing can be done so that the glued solution exactly matches the original one outside the gluing region,
so long as certain nondegeneracy conditions (``no KIDS") hold at the points of gluing. When this can be done, the gluing is said to be \textit{localized}.
More generally,  the glued solutions may not exactly match the original ones outside the gluing region, but can be constructed so that the data set is  arbitrarily close to the original solution away from this region; this is called \textit{non-localized} gluing.

Say one chooses a pair of (disjoint) asymptotically hyperbolic solutions of the constraints and glues them at a pair of points satisfying the necessary conditions, as described in either \cite {IMP1} or \cite{CIP}. If each of the original AH data sets has a single (connected) asymptotic region (as described below in Section \ref{sec:AHID}), then the glued data set, which is also asymptotically hyperbolic,  necessarily has two disjoint asymptotic regions. If we are working with AH  initial data sets which are viewed as data on partial Cauchy surfaces that intersect null infinity in an asymptotically simple spacetime  \cite{Wald}, then the existence of multiple asymptotic regions is problematic for physical modeling.

In the present paper, we show that one can glue asymptotically hyperbolic solutions of the constraint equations in such a way that in fact the asymptotic region of the glued data set is connected. The idea, which is modeled after the studies  of Mazzeo and Pacard on gluing asymptotically hyperbolic Einstein manifolds \cite{rafeandpacard}, is to use the conformally compactified representation of asymptotically hyperbolic geometries, which models a complete, asymptotically hyperbolic manifold as the interior of a compact manifold with boundary (see Section \ref{sec:AHID} below).  The boundary of this manifold, which we call the \emph{ideal boundary}, 
is not part of the physical initial manifold, but represents asymptotic directions at infinity. 
The gluing is done using points $p_1$ and $p_2$ lying on the ideal boundary.  In this context, we will use the term \emph{asymptotic region} to refer to any open collar neighborhood of the boundary, with the boundary itself deleted.
If the original manifold has a single connected asymptotic region, then so does the glued manifold.
(See Fig.\ \ref{fig:1}.)
\begin{figure}[h]
\begin{minipage}[b]{.4\textwidth}
\begin{center}
{\includegraphics{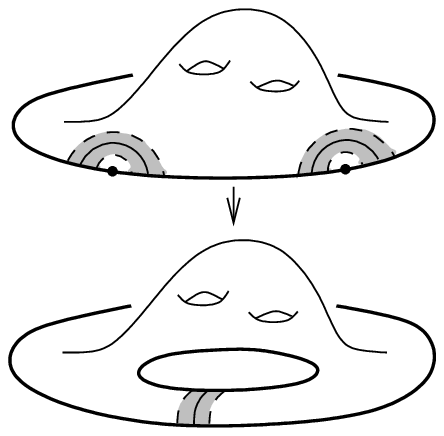}}
\caption{}\label{fig:1}
\end{center}
\end{minipage}
\begin{minipage}[b]{.59\textwidth}
\begin{center}
{\includegraphics{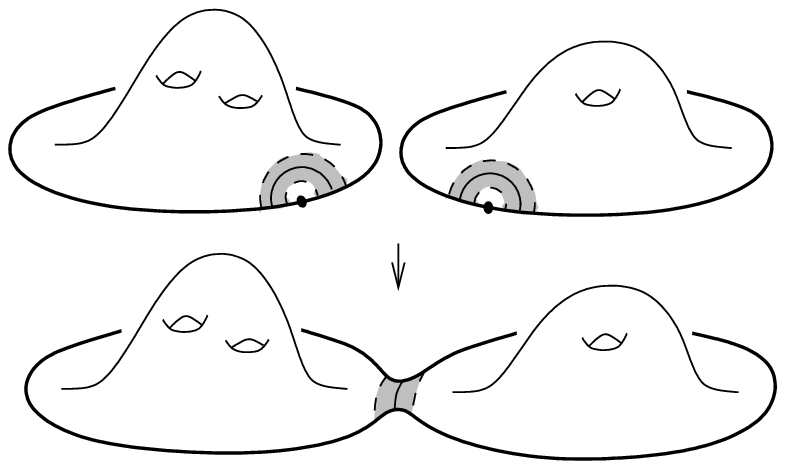}}
\caption{}\label{fig:2}
\end{center}
\end{minipage}
\end{figure}
If the original manifold has two disjoint 
connected 
asymptotic regions with one boundary point chosen on
the ideal boundary of each region,
then the glued manifold will have a single connected asymptotic region
(Fig.\ \ref{fig:2}).
(Although the ideal boundary in Fig.\ \ref{fig:1} appears to be disconnected, in the cases of interest the boundary of the glued manifold will be a connected sum of connected 
$2$-manifolds, which is always connected.)

The results we present here do not hold for general AH initial data sets. We require that the data have constant mean curvature (``CMC"), in the sense that $\Tr_g K$ is constant on $M$. Also, our results thus far provide sufficient conditions for (asymptotic) non-localized gluing. In a future paper, we hope to both eliminate the CMC restriction, and find conditions which are sufficient for the gluing to be localized.

To set up our work here, we start in Section \ref{sec:prelim} with a definition and discussion of asymptotically hyperbolic initial data and their polyhomogeneous behavior 
in the asymptotic region. The section continues with a brief description of the conformal method for generating solutions of the constraint equations  and the   simplifications of the method which occur for constant mean curvature data. 
The conformal method is discussed here because it plays an important role both in constructing basic examples of AH initial data \cite{piot} and in carrying out our gluing procedure. We conclude Section \ref{sec:prelim} with an overview of those aspects of \cite{piot} on which our work relies. We state our main theorem in Section \ref{sec:main-thm}. We then carry out the first part of the gluing construction (``splicing") in Section \ref{sec:IMP}, producing a family of initial data sets depending on a small parameter $\e$,
which satisfy the constraint equations approximately. Our splicing construction is modeled after the 
construction of asymptotically hyperbolic Einstein metrics
in  \cite{rafeandpacard}. 

The proof of the main theorem relies on the use of weighted H\"older spaces. We define and discuss these spaces in Section \ref{sec:analysis}, following \cite{lee}. In the rest of Section \ref{sec:analysis}, we study certain elliptic operators which act on weighted H\"older spaces, and prove that they are invertible with norms of their inverses bounded uniformly in $\e$. One of the key steps here  is a blow-up analysis argument similar to that of   \cite{rafeandpacard}. Our analytical results are applied in Section \ref{sec:correcting-II} to correcting the traceless part of the glued second fundamental form. Finally, in Section \ref{lish} we use results of \cite{GL} and a contraction mapping argument to solve the Lichnerowicz equation and complete the proof of the main theorem. Section \ref{sec:concl} contains concluding remarks.

\section{Preliminaries}\label{sec:prelim}

In this section we define and discuss examples of asymptotically hyperbolic initial data sets. We also review the conformal method for creating solutions of the constraint equations. 

\subsection{Asymptotically Hyperbolic Initial Data Sets}\label{sec:AHID}
The model for asymptotically hyperbolic initial data sets is the
hyperboloid in
$(3+1)$-dimensional Minkowski space,
\begin{equation*}
\breve M = \{(x^0,x^1,x^2,x^3): x^0>0 \text{ and } (x^0)^2 = (x^1)^2 + (x^2)^2 + (x^3)^2+1\},
\end{equation*}
together with
its induced Riemannian metric $\breve g$
(which is a model for the hyperbolic metric of constant sectional curvature $-1$)
and its extrinsic curvature $\breve K\equiv \breve g$; it is straightforward to check that
$(\breve M,\breve g, \breve K)$ satisfies
\eqref{1steq} and \eqref{2ndeq}.

Roughly speaking, an asymptotically hyperbolic
initial data set $(M,g,K)$ is one in which the Riemannian metric $g$ is
complete and approaches constant negative curvature as one approaches the ends of the manifold, and
the extrinsic curvature $K$ asymptotically approaches
a pure trace tensor field that is a constant multiple of the metric.

A more precise definition of asymptotic hyperbolicity is motivated by the Poincar\'e disk model of hyperbolic space.
In that model, hyperbolic space is given by the (smooth) metric
\begin{equation}
g=\frac{4}{(1-|x|^2)^2}\left( (dx^1)^2 + \dots +(dx^n)^2\right),
\end{equation}
on the open unit ball.
If we define the function $\rho =\frac{1}{2}(1-|x|^2)$, then
$\overline{g}= \rho^2 g$ is the Euclidean metric and thus extends smoothly to the
closure of the open ball.
Note that the function $\rho$ is smooth on the closed ball, it picks out the boundary of the ball since $\rho^{-1}(0)$ is equal to this boundary, and it satisfies the derivative condition $|d\rho|_{\overline{g}}=\big|\frac{1}{2}d\big(1-|x|^2\big)\big|_{\overline{g}}=1$ on the boundary of the ball.

With this example in mind,
we
make the following definitions.
We suppose throughout this paper that
$M$ is the interior of a smooth, compact manifold with boundary $\Mbar$.
A \emph{defining function} for $\Mbar$ is a
nonnegative real-valued function $\rho: \Mbar \rightarrow \mathbb{R}$ of class at least $C^1$
such that
$\rho^{-1}(0)=\partial \Mbar$ and $d\rho$ does not vanish on $\partial\Mbar$.
Given a nonnegative integer $k$ and a real number $\alpha\in [0,1]$,
a smooth Riemannian metric $g$ on $M$
is said to be \textit{conformally compact of class $C^{k}$}
(or $C^{k,\alpha}$, or $C^\infty$)
if there exists a smooth defining function $\rho$ and a
Riemannian metric
$\overline{g}$ on $\Mbar$ of class
$C^k$, $C^{k,\alpha}$, or $C^\infty$, respectively,
with $g=\rho^{-2} \overline{g}|_{M}$.
\emph{Smoothly conformally compact} means the same as conformally
compact of class $C^\infty$.

If $g$ is conformally compact and if in addition
$|d\rho|_{\overline g} = 1$ on $\partial\Mbar$, then $g$ is said to be
\emph{asymptotically hyperbolic} (of class $C^k$, $C^{k,\alpha}$, or $C^\infty$,
as appropriate).
One verifies easily that any asymptotically hyperbolic metric of
class at least $C^2$
has sectional curvatures approaching $-1$ near $\partial\Mbar$
(see \cite{Mazzeo-edge}), and so
indeed has the intuitive properties mentioned above.
The boundary $\partial\Mbar$ is called the
\emph{ideal boundary}, and $\partial\Mbar$ together with its induced metric
$\iota^*\overline g$ (where $\iota\colon \partial\Mbar\hookrightarrow\Mbar$ is inclusion),
is called \textit{conformal infinity}.
Note that for a given AH metric, the choice of the defining function is not unique,
so the geometry of the conformal infinity is only defined up to a conformal factor.

Naively, one might hope to
work with initial data sets $(M,g,K)$ in which $(M,g)$ is a
smoothly asymptotically hyperbolic Riemannian manifold, and $K$ is a
symmetric $2$-tensor field such that $\rho^2 K$ has a smooth extension to $\Mbar$ and
is equal to a constant multiple of the metric on the ideal boundary.
Unfortunately, however, it is shown in
\cite{piot2} that there are obstructions to finding solutions
with this degree of smoothness,
marked by the presence of log terms in the asymptotic expansions of $g$ and $K$
near the ideal boundary.
For this reason, instead of smoothness  we have to settle for a slightly weaker
notion called polyhomogeneity, which we now define.

A function $f\colon M\to \R$ is said to be \emph{polyhomogeneous} (cf.\ \cite{Mazzeo-edge})
if it is smooth in $M$, and there
exist a sequence of real numbers $s_i \nearrow +\infty$, a sequence of
nonnegative integers $\{q_i\}$, and  functions $\overline f_{ij}\in C^\infty(\Mbar)$  such that
\begin{equation}\label{eq:def-polyhomo}
f \sim \sum_{i=1}^\infty \sum_{j=0} ^{q_i} \rho^{s_i} (\log \rho)^j \overline f_{ij}
\end{equation}
in the sense that
for any positive integer $K$, there exists a positive
integer $N$ such that the difference
\begin{displaymath}
f -
\sum_{i=1}^{N}\sum_{j=0}^{q_i} \rho^{s_i} (\log \rho)^j \overline f_{ij}
\end{displaymath}
is $O(\rho^K)$ as $\rho\to 0$, and remains $O(\rho^K)$ after  being differentiated any number of
times by smooth vector fields on $\Mbar$ that are tangent to $\partial \Mbar$.
It is easy to check that sums and products of polyhomogeneous functions are
polyhomogeneous, as are quotients of polyhomogeneous functions provided
that the denominator has no log terms with its lowest power of $\rho$
(i.e., $q_1=0$) and provided that its leading term $\overline f_{10}$ does not vanish on the
ideal boundary.
A tensor field on $M$ is said to be polyhomogeneous if it is smooth
on $M$ and its component
functions are polyhomogeneous
in some smooth coordinate chart  in a neighborhood of
every ideal boundary point.
We define a
\emph{polyhomogeneous asymptotically hyperbolic Riemannian metric} on $M$
to be a polyhomogeneous Riemannian metric $g$ which is also conformally
compact of class at least $C^2$.

Now we come to the main definition of this section.

\begin{definition}\label{def:phg-dataset}
A \textit{polyhomogeneous asymptotically hyperbolic
initial data set} (sometimes called a \textit{hyperboloidal} initial data set)
is a triple $(M, g, K)$, in which
\begin{enumerate}\romanletters
\item
$M$ is the interior of a smooth, compact manifold with boundary $\Mbar$;
\item
$g$ is a polyhomogeneous AH Riemannian metric on $M$;
\item
$K$ is a polyhomogeneous symmetric covariant $2$-tensor field on $M$ with the
property that for any smooth defining function $\rho$,
$\rho^2 K$ has a $C^2$ extension to $\Mbar$ whose
restriction to $\partial\Mbar$ is
a constant multiple of (the extension of) $\rho^2 g$ there;
\item
the Einstein constraint equations \eqref{1steq} and \eqref{2ndeq} are satisfied.
\end{enumerate}
\end{definition}
This definition implies that $g$ and $K$ can be written in the form
\begin{align}
\label{eq:g-asymp}
g &= \rho^{-2} \overline g,\\
\label{eq:K-asymp}
K &= \frac{\tau}{n} g + \rho^{-1}\overline\mu,
\end{align}
where $\overline g$ is a polyhomogeneous Riemannian metric on $\Mbar$ that is
of class at least $C^2$ (and thus has log terms, if any, only with powers of
$\rho$ greater than $2$); $\tau$ is a polyhomogeneous $C^2$ scalar function on $\Mbar$ whose
restriction to $\partial\Mbar$ is constant; and $\overline\mu$ is a
polyhomogeneous symmetric $2$-tensor field on $\Mbar$ that is trace-free with
respect to $\overline g$ and has log terms only with powers of $\rho$
greater than $1$. 
A polyhomogeneous AH initial data set is defined to be \emph{CMC} if the mean curvature function $\tau=\Tr_g K$ is constant. 
We note that for such data $\tau$ must be identically equal to $3$. Indeed, 
the definition of CMC asymptotically  hyperbolic data
guarantees that the scalar curvature approaches the constant $-n(n-1) = -6$ at the
ideal boundary. Since the $g$-norm of $\rho^{-1}\overline \mu$ 
approaches zero at the ideal boundary,
inserting
\eqref{eq:K-asymp} into \eqref{2ndeq} and evaluating in the limit at the ideal boundary
implies that
the mean curvature $\tau$ must be identically equal to $3$.
Thus \eqref{eq:K-asymp} becomes
\begin{equation*}
K = g+\mu = \rho^{-2}\overline g + \rho^{-1} \overline\mu.
\end{equation*}
It is shown by Andersson and Chru\'sciel in \cite{piot} that CMC polyhomogeneous AH  initial data sets exist in abundance.

\subsection{The Conformal Method for Finding Solutions of the Constraint Equations}\label{sec:conf-method}
Both the construction of \cite{piot} and our gluing construction here are based on the conformal method, which (along with the closely related conformal thin sandwich method) is the most widely used method for producing solutions of the constraint equation. We proceed by reviewing this method.  

We start by introducing two auxiliary differential operators which are involved  in the conformal method. The first of the two is the 
{\it conformal Killing operator} $\mathcal{D}_{\lambda}$ which acts on vector fields $X$ as follows:
\begin{equation}\label{def:conf-killing}
(\mathcal{D}_{\lambda} X)_{cd}:=\frac{1}{2} \Lie_X \lambda_{cd}-\frac{1}{3}(\mathrm{div}_{\lambda}X)\lambda_{cd} = \frac{1}{2}( \nabla_c X_d +\nabla _d X_c )- \frac{1}{3} \nabla_a X^a \lambda_{cd}.
\end{equation}
The image of $\mathcal{D}_\lambda$ is contained in the space of symmetric $2$-tensors which  are traceless with respect to $\lambda$. The formal adjoint 
of $\mathcal{D}_{\lambda}$ is
\begin{equation}\label{adjoint}
\mathcal{D}_{\lambda}^*T=-\big(\divergence_{\lambda}T)^{\sharp};
\end{equation}
here (and throughout the paper) the symbol $^{\sharp}$ refers to raising an index. Another auxiliary  operator we use is the elliptic, formally self-adjoint \emph{vector Laplacian operator} 
$$\vlap_{\lambda} X=(\mathcal{D}_{\lambda}^*\circ \mathcal{D}_{\lambda})X=-\left(\divergence_\lambda \left(\mathcal{D}_\lambda X\right)\right)^\sharp.$$

The conformal method for $3$-dimensional manifolds (we restrict to $n=3$ for convenience) is based on the 
\emph{Lichnerowicz-York decomposition} 
of data \cite{CBY}
\begin{align}
\label{gdecomp}
g_{ab} &= \psi^4 \lambda_{ab},\\
\label{Kdecomp}
K_{cd} &= \psi^{-2}(\nu_{cd} + 2 (\mathcal{D}_{\lambda} W)_{cd}) +\frac{1}{3}\psi^4 \tau \lambda_{cd},
\end{align}
where $\lambda_{ab}$ is a Riemannian metric, $\nu_{cd}$ is a symmetric tensor field that is divergence free and trace free with respect to $\lambda_{ab}$, $\tau$ is a scalar function, $\psi$ is a positive definite scalar function and $W^c$ is a vector field. Substituting the field decompositions \eqref{gdecomp}--\eqref{Kdecomp}
into the vacuum
constraint equations (\ref{1steq})--(\ref{2ndeq}) and
using standard conformal transformation formulas for the scalar curvature and for divergences, we obtain
\begin{align}
\label{confmomeq}
(\vlap_{\lambda}W)_c &= \frac{1}{3}\psi^6 \nabla_c \tau,\\
\label{Lichneroeq}
\Delta_{\lambda} \psi &= \frac{1}{8} R(\lambda) \psi - \frac{1}{8} |\nu_{cd} + 2 (\mathcal{D}_{\lambda} W)_{cd}|^2_{\lambda} \psi^{-7} + \frac{1}{12} \tau^2 \psi^5.
\end{align}
(Here $\Delta_\lambda$ denotes the Laplace-Beltrami operator with respect to the metric $\lambda$.)

The idea of the conformal method is to choose any \emph{conformal data}  $(M, \lambda, \nu, \tau)$ in which $\nu$ is traceless and divergence-free with respect to $\lambda$, and then use the coupled PDE system (\ref{confmomeq})--(\ref{Lichneroeq}) to solve for the determined data $(\psi, W)$. If, for a given set of conformal data,  one can solve the system (\ref{confmomeq})--(\ref{Lichneroeq}), then the initial data fields obtained by recomposing the fields as in (\ref{gdecomp})--(\ref{Kdecomp}) provide a solution to the constraint equations.

To execute such a construction of initial data one needs to have a symmetric, traceless,
 and divergence-free tensor $\nu$. There is a standard method for finding such a tensor 
\cite{York}.
The idea behind the method is to start with an arbitrary traceless symmetric $2$-tensor field $\mu$ and then find a vector field $X$ which satisfies 
\begin{equation}\label{babyL}
L_\lambda X=(\divergence_{\lambda} \mu)^{\sharp}.
\end{equation}
Using \eqref{adjoint} one easily verifies that 
$\nu:=\mu+\mathcal{D}_\lambda X$ 
is symmetric, traceless and divergence-free.

There have been extensive studies to determine which sets of conformal data lead to solutions, and which do not. (See  \cite{BI} for a recent review.) This issue is best understood for conformal data with constant $\tau$, which leads to initial data with constant mean curvature (``CMC"). In this case, the constraint equations (\ref{confmomeq})--(\ref{Lichneroeq}) effectively decouple---the (unique) solution to (\ref{confmomeq}) is $\mathcal{D}_{\lambda}W=0$---and one need only analyze the solvability of the (remaining) Lichnerowicz equation
\begin{equation}
\label{LichneroeqW=0}
\Delta_{\lambda} \psi - \frac{1}{8} R(\lambda) \psi + \frac{1}{8} |\nu_{cd}|^2_{\lambda} \psi^{-7} - \frac{1}{12} \tau^2 \psi^5=0.
\end{equation}
It is also fairly well understood for ``near CMC" conformal data sets, which are characterized by $|\nabla \tau|_{\lambda}$ being sufficiently small.

\subsection{Conformal method and Andersson-Chru\'sciel initial data}\label{AC}
The Andersson-Chru\'sciel construction 
\cite{piot} of polyhomogeneous AH CMC initial data on a smoothly conformally compact 3-dimensional manifold $M$ starts by choosing a smoothly conformally compact metric $\lambda$ on $M$ and
a traceless,  symmetric $2$-tensor $\mu=\rho^{-1}\bar{\mu}$ for which $\bar{\mu}\in C^\infty(\Mbar)$. The construction continues by  finding a solution $X$ of  the vector Laplacian equation \eqref{babyL} and by considering the symmetric, traceless, divergence-free tensor 
$\nu:=\mu+\mathcal{D}_\lambda X$.
The polyhomogeneity of $X$, and consequently 
of $\nu$, arises naturally here as a consequence of the indicial roots of the vector Laplacian (for details on indicial roots  see \cite{lee}). More precisely, we have the asymptotic expansions
\begin{align}
X &\sim \rho^2 \overline X_0+(\rho^4\log\rho)\overline X_1 &&\text{with $\overline X_0, \overline X_1 \in C^\infty(\Mbar)$},
\label{eq:X-expansion}
\\
\nu &\sim
\rho^{-1} \overline \nu_0 + (\rho\log\rho)\overline\nu_1 &&\text{with $\overline\nu_0, \overline\nu_1\in C^\infty(\Mbar)$}.
\label{eq:nu-expansion}
\end{align}
Note that the description in \cite{piot} treats $\nu$ as a contravariant $2$-tensor,
which accounts for the difference between our powers of $\rho$ and the powers of $\Omega$ and $x$ in \cite{piot}.
It is useful that 
Andersson-Chru\'sciel \cite{piot} also prove a sequence of existence and uniqueness results regarding solutions of equations such as \eqref{babyL} in the context of polyhomogeneous tensor fields; we rely on these results when we conclude that the 
perturbations we make are polyhomogenous. 

The construction in \cite{piot} proceeds with the 
analysis of the Lichnerowicz equation. It is shown that the boundary value problem consisting of the Lichnerowicz equation \eqref{LichneroeqW=0} with $\tau =3$ and the boundary condition 
$\psi\big|_{\partial M}=1$
has a polyhomogeneous solution of the form 
\begin{equation}\label{eq:psi-expansion}
\psi \sim
1 + \rho\overline \psi_0 + \sum_{i=1}^\infty (\rho^{3}\log\rho)^i \overline\psi_i,
\qquad\text{with $\overline\psi_i\in C^\infty(\Mbar)$}.
\end{equation}
The exponent in $\rho^{3}\log\rho$ corresponds to the indicial root of the linearization of the Lichenrowicz operator (cf.\ the left-hand side of \eqref{LichneroeqW=0} with $\tau=3$) in the neighborhood of the constant function $\psi_0\equiv 1$. We point out that the analysis in \cite{piot} also includes an existence and uniqueness result for the 
Lichnerowicz boundary value problem  with  polyhomogeneous data. We need this result when we show  that our solution of the Lichnerowicz equation is polyhomogeneous (see Theorem \ref{thm:Lichnerowicz-main}).

For convenience,
 we say that a polyhomogeneous AH initial data set $(M,g,K)$ is of 
\emph{Andersson--Chru\'sciel  \(A--C\) type} if it can be written in the form 
$g = \psi^4 \lambda$ and $K = g+\psi^{-2}\nu$, in which 
$\lambda$ is a smoothly conformally compact metric, 
$\nu$ is 
a symmetric $2$-tensor field that is divergence free and trace free with respect to $\lambda$
and has an asymptotic
expansion of the form \eqref{eq:nu-expansion},
and $\psi$ is a positive function with an asymptotic expansion of the form \eqref{eq:psi-expansion}.
The discussion in \cite[Appendix A]{piot} shows that 
generically, initial data of A--C type are
the ``smoothest possible'' AH initial data.  

Throughout this paper, 
we assume only
that the asymptotically hyperbolic initial data sets 
we work with are polyhomogenous in the sense of Definition \ref{def:phg-dataset}
(which includes data of A--C type as a special case).
Our gluing procedure 
then produces new data of the same type.
We note that if our starting data set is of A--C type, it will not generally follow from our main theorem 
(Theorem \ref{thm:main} below) that the solution we obtain after gluing is also of A--C type; our results 
will only guarantee that this solution is polyhomogeneous in the sense of Definition \ref{def:phg-dataset}. 
The difficulty is that while A--C data is obtained by solving the conformal constraints with smooth 
conformal data, in carrying out the gluing we must solve these equations for conformal data which includes log terms.

\section{Main Gluing Theorem}\label{sec:main-thm}

With the conventions established above, we are ready to state our main theorem.
For convenience of exposition, we focus our attention here on the case $\dim M=3$.  
We have little 
doubt that the theorem and its proof generalize easily to higher dimensions.  

\begin{theorem}\label{thm:main}
Let $(M, g, K)$ be a polyhomogeneous asymptotically hyperbolic CMC
initial data set \(with $\dim M=3$\)
that satisfies the Einstein \(vacuum\) constraint
equations, and let $p_1,p_2$ be distinct points in the ideal boundary 
$\partial\Mbar$.
Then for each $\e > 0$ there exists a
polyhomogeneous AH initial data set
$(M_\e, g_{\e}, K_{\e})$  
such that
\begin{enumerate}\romanletters
\item
$M_\e$ is diffeomorphic to the interior of a boundary connected sum, obtained from
$\overline{M}$ by excising small half-balls $B_1$ around
$p_1$ and $B_2$ around $p_2$, and identifying their boundaries.
\item
$(M_\e, g_{\e}, K_{\e})$ is a solution to the vacuum constraints. 
\item
On the complement of any fixed small half-balls surrounding $p_1$ and $p_2$ in $M$,
and away from the
corresponding neck region in $M_\e$, the data
$(g_{\e}, K_{\e})$ converge uniformly in 
$C^{2,\alpha}\times C^{1,\alpha}$ to $(g,K)$,
for some
$\alpha\in(0,1)$.
\end{enumerate}
\end{theorem}

In fact, the convergence of $K_\e$ is a little better than $C^{1,\alpha}$: away from the fixed half-balls, it actually converges in a weighted $C^{1,\alpha}$
space.  See Theorem \ref{thm:nu-convergence} for the precise statement. 

Note that, as is the case for the non-localized gluing of AH data sets at interior points (see \cite{IMP1}), there is no need to impose any nondegeneracy 
conditions on the data in the neighborhood of the gluing points $p_1$ and $p_2$.

\section{Splicing Construction}\label{sec:IMP}

We presume 
that we are given a $3$-dimensional CMC
polyhomogeneous asymptotically hyperbolic
initial data set
$(M, g, K)$, which is not assumed to be connected.
We let $\rho$ denote a chosen smooth defining function for $\overline M$.
We may write 
\begin{equation*}
K = g+\mu = \rho^{-2}\overline g + \rho^{-1} \overline\mu,
\end{equation*}
where 
$\overline g$ and $\overline \mu$ are polyhomogeneous,
$\overline g\in C^2(\Mbar)$, $\overline\mu\in C^1(\Mbar)$, and 
$\overline\mu$ is trace-free with respect to $\overline g$ (or, 
equivalently, $g$).  
Note that the assumption that $\overline g$ is polyhomogeneous and of class $C^2$
means that the first log term in the expansion of $\overline g$ must occur with a power of $\rho$
strictly greater than $2$, and thus $\overline g$ is actually
in $C^{2,\alpha}(\Mbar)$ for some $\alpha\in (0,1)$; and similarly
$\overline\mu\in C^{1,\alpha}(\overline M)$.

The gluing construction is a step by step procedure. We outline the main steps here:
The first step, which we call \emph{splicing},
involves the construction of a one-parameter family of manifolds and initial data sets that are CMC and polyhomogeneous AH, but that only approximately solve the constraint equations. 
(In most of the literature discussing gluing constructions, both the first step leading to approximate solutions and the complete construction leading to exact solutions are called ``gluing."
Here, to distinguish the two, we call the procedure leading to the approximate solutions ``splicing.")

The new manifolds $M_{\e}$ are obtained by a connected sum construction which is executed in the \emph{preferred background coordinates}. 
The parameter $\e$ labels the coordinate ``size" of the ``bridge", or gluing region. Next we use cutoff functions tied to the parametrized gluing region to construct a parametrized set of metrics $g_{\e}$ on $M_{\e}$. To verify that these spliced metrics are all asymptotically hyperbolic, we also construct a parametrized set of spliced defining functions $\rho_{\e}$. Using a different cutoff procedure, we produce a family of (spliced) symmetric $2$-tensors
$\mu_{\e}$ that, by construction, are trace free with respect to the corresponding $g_{\e}$, but are generally \textit{not} divergence free with respect to $g_{\e}$. The next two steps involve deformations of the spliced data sets $(M_{\e}, g_{\e}, K_{\e}=g_{\e}+\mu_{\e})$ to produce the glued data sets which satisfy the constraints and have the desired limit properties. To deform $\mu_{\e}$, we first estimate its divergence, and then (following the standard York prescription) solve a linear elliptic system (based on the vector Laplacian $\vlap_{g_\e}$) whose solution tensor deforms $\mu_{\e}$ to a new family of tensors $\nu_{\e}$ that are divergence free. To deform the metric, we treat $(M_{\e}, g_{\e}, \nu_{\e}, 3)$ as a set of CMC conformal data, and proceed to solve the Lichnerowicz equation (\ref{LichneroeqW=0}) for a family of conformal factors $\psi_{\e}$. The 
$\e$-parametrized data sets $(M_{\e}, \psi^4_{\e} g_{\e}, \psi^{-2}_{\e} \nu_{\e} +  \psi^4_{\e}g_{\e})$, which we call the glued data, then solve the constraint equations, and are verified to approach arbitrarily close (as $\e \rightarrow 0)$ to the original data away from the gluing region.

In the rest of this section, we detail the splicing constructions. We detail the deformation steps in subsequent sections.

\subsection{Preferred background coordinates}
We focus first on the given polyhomogeneous
asymptotically hyperbolic geometry $(M, g)$.
Let $\rho$ be a smooth defining function, and
define $\overline g = \rho^2 g$, which is a $C^{2,\alpha}$ polyhomogenous Riemannian
metric on $\Mbar$.
For each point $p\in \partial \Mbar$, we can choose smooth functions $\theta^1,\theta^2$ such that
$(\rho,\theta^1,\theta^2)$ form smooth coordinates
in a  neighborhood $\U\subset\Mbar$,
which we call \emph{background coordinates}.
Sometimes for reasons of notational symmetry we also set $\theta^0=\rho$.
Throughout this paper, we will index such background coordinates with indices named
$a,b,c,\dots$,
which we understand to run from $0$ to $2$; and we will use indices $j,k,\dots$,
running from $1$ to $2$, to
refer to coordinates on $\partial \Mbar$.
We will use the Einstein summation convention when convenient.

It is shown in  \cite{GL} that when $g$ is an asymptotically hyperbolic metric that is smoothly
conformally compact,
there is a smooth defining function $\rho$ such that
$|d\rho|_{g}^2/\rho^2\equiv 1$
in a neighborhood of the ideal boundary $\partial \Mbar$, so
the metric can be written in the form
$g=\rho^{-2}\bigl(d\rho^2+h(\rho)\bigr)$ there,
where $h(\rho)$ is a smoothly-varying
family of metrics on $\partial \Mbar$ with $h(0)=\overline{g}\big{|}_{T\partial \Mbar}$.
Unfortunately, that result does not apply in the present circumstances because we are
not assuming that $g$ has a smooth conformal compactification.  As a substitute,
however, we have the following lemma:

\begin{lemma}\label{def-fcn}
If $(M,g)$ is a polyhomogeneous asymptotically hyperbolic Riemannian
geometry, then there exists
a smooth defining function $\rho$
such that
\begin{equation}\label{drho}
\frac{|d\rho|^2_{g}}{\rho^2} = 1+O(\rho^2).
\end{equation}
Also, for each $p\in\partial \Mbar$ there exist smooth background coordinates
$(\rho,\theta^1,\theta^2)$ on
an open neighborhood $\U$ of $p$ in $\overline M$
in which
$g$ can be written in the form
\begin{equation}
\label{grahamlee}
g=\rho^{-2} \Bigl(d\rho^2 +(d\theta^1)^2 + (d\theta^2)^2+ m_{ab}(\rho,\theta)d\theta^a\,d\theta^b\Bigr),
\end{equation}
where the ``error terms'' $m_{ab}$ are uniformly bounded in $\U$
and satisfy
\begin{align}
m_{00}(\rho,\theta) = m_{j0}(\rho,\theta) = m_{0j}(\rho,\theta) & = O(\rho^{2}), &j\in \{1,2\},\label{ma0-est}\\
m_{jk}(\rho,\theta) & = O\bigl(\rho + (\theta^1)^2 + (\theta^2)^2\bigr), & j,k\in\{1,2\}.\label{mjk-est}
\end{align}
Moreover,
$g$ is uniformly equivalent in $\U$ to
the
metric $\rho^{-2} \bigl(d\rho^2 +(d\theta^1)^2 + (d\theta^2)^2\bigr)$.
\end{lemma}

\begin{proof}
Let $\rho_0$ be any smooth defining function for $\Mbar$, and
write $\overline g_0 = \rho_0^2 g$.
The hypothesis
implies that there is a smooth
metric $\overline g_1$ on $\Mbar$ such that $\overline g_0 =  \overline g_1 + O(\rho^2)$.
(Just take $\overline g_1$ locally to be equal to the leading smooth terms in an asymptotic
expansion for $\overline g_0$, and then patch together with a partition of unity.)
Let $g_1 = \rho_0^{-2}\overline g_1$.
Because $g_1$ is asymptotically hyperbolic and
smoothly conformally compact, the argument of \cite{GL} shows that
there is a smooth defining function $\rho$ such that
$|d\rho|^2_{ g_1} /\rho^{2}\equiv 1$.
It follows that $|d\rho|^2_{g}/\rho^2 = 1+O(\rho^2)$.
Let $\widehat g=\iota^*\overline g_0= \iota^*\overline g_1$ be the
metric induced on $\partial\Mbar$ by inclusion.  

Given $p\in \partial\Mbar$,
let $(\theta^1,\theta^2)$ be
Riemannian normal coordinates for
$\widehat g$
on some neighborhood of $p$ in $\partial\Mbar$.
Extend $(\theta^1,\theta^2)$ to a neighborhood of $p$ in $\Mbar$
by declaring them to be constant
along the integral curves of the smooth vector field
$\grad_{\overline g_1}\rho$.
It follows that $(\rho, \theta^1,\theta^2)$ are smooth coordinates in a neighborhood $\U$ of
$p$, in which $\overline g_1$ has an expression of the form
\begin{equation*}
\overline g_1=d\rho^2 +(d\theta^1)^2 + (d\theta^2)^2+ m_{ab}(\rho,\theta)d\theta^a\,d\theta^b,
\end{equation*}
with $m_{00}$, $m_{j0}$, and $m_{0j}$ identically zero, and $m_{jk}$ satisfying
\eqref{mjk-est}.
Because $g=\rho^{-2}\bigl(\overline g_1+O(\rho^2)\bigr)$, this implies that $g$ has the
expansion claimed in the statement of the lemma.
Since $m_{ab}(0,0)=0$, by shrinking $\U$ we may also
ensure that the coefficients $m_{ab}$ are uniformly small in $\U$, and thus
$g$ is uniformly equivalent to
$\rho^{-2} \bigl(d\rho^2 +(d\theta^1)^2 + (d\theta^2)^2\bigr)$ there.
\end{proof}

From now on, we  assume $\rho$ is a smooth defining function
satisfying \eqref{drho}.
We now argue that we can choose $\rho$ so that it also satisfies 
\begin{equation}\label{delta-rho}
\Delta_g \rho \le0 \text{\ \ on\ \ } M.
\end{equation}
On any asymptotically hyperbolic $3$-manifold, an easy computation 
(see, for example, 
\cite[p. 199]{GL})
shows that 
\begin{equation*}
\frac{\Delta_g \rho}{\rho}\to -1 \ \ \text{as}\ \ \rho\to 0.
\end{equation*}
Thus, there is some $\delta>0$ such that $\Delta_g\rho \le 0$ 
on the set where 
$\rho<\delta$. Let $\sigma\colon [0,+\infty)\to \left[0,\frac{3}{4}\delta\right]$ be any smooth,
increasing, concave-down function for which
$$\sigma (x)=x \ \ \text{if}\ \ x\le \delta/2, \ \ \ \sigma(x)=\frac{3}{4}\delta \ \ \text{if}\ \ x\ge \delta.$$
Define $\tilde{\rho}:=\sigma\circ \rho$. Note that the conclusions of the previous lemma still hold if the function $\rho$ 
is replaced by $\tilde{\rho}$. Furthermore, we compute
$$\Delta_g \tilde{\rho}=\divergence_g \bigl((\sigma'\circ\rho)d\rho\bigr)=\left(\sigma'\circ\rho\right)\Delta_g\rho+\left(\sigma''\circ \rho\right) |d\rho|_g^2\le 0,$$
where the last inequality follows from the facts that $\sigma'\ge 0$, $\sigma''\le 0$, and
$\Delta_g\rho<0$ on the support of $\sigma'\circ\rho$. 
From now on, we replace $\rho$ by $\tilde\rho$, and assume that \eqref{delta-rho} holds.

We  call any coordinates $(\rho,\theta^1,\theta^2)$
that satisfy the conclusions of the previous lemma and \eqref{delta-rho}
\emph{preferred background coordinates centered at $p$}.

\subsection{Splicing the manifolds and the metrics}
We now focus on the topological aspect of our gluing construction. Let $p_1, p_2\in \partial \Mbar$
be two distinct points on the ideal boundary,  and for $i=1,2$ let
$\vec\theta_i = (\rho,\theta^1_{i},\theta^2_{i}) = (\theta^0_i,\theta^1_i,\theta^2_i)$
be preferred background coordinates on a neighborhood $\U_i\subset\Mbar$ centered at $p_i$.
There is a positive constant $c$ such that these preferred coordinates
are defined and \eqref{grahamlee}--\eqref{mjk-est} hold for
$|\vec{\theta}_i|\le c$; after multiplying
$\rho$ and each of the coordinate functions $\theta^j_i$
by $1/c$ (which does not affect \eqref{drho}, \eqref{grahamlee} or \eqref{delta-rho}), we may assume that
these two preferred coordinate charts are defined for $|\vec{\theta}_i|\le 1$.

We now let $\e$ be a small positive parameter, and
consider  two
``semi-annular" regions $\overline A_{\e,1},\overline A_{\e,2}\subset \Mbar$ characterized by
\begin{equation*}
\overline A_{\e,i} := \left\{ \vec{\theta}_{i}\in \U_i:
\e^2<|\vec{\theta}_i|< 1\right\}.
\end{equation*}
We let $A_{\e,i} = \overline A_{\e,i}\cap M$.
For each choice of $\e$, the two regions can be identified using an inversion map with respect to a circle of radius $\e$,
given explicitly in coordinates by
$\vec\theta_2 = I_\e (\vec\theta_1)$, where $I_\e\colon \overline A_{\e,1}\to \overline A_{\e,2}$ is
the following diffeomorphism:
\begin{equation}\label{inv}
I_\e (\vec{\theta}_1)=\frac{\e^2}{|\vec{\theta}_1|^2}\vec{\theta}_1.
\end{equation}
Based on this map, we define an equivalence relation  on $\Mbar$
by saying $\vec\theta_1\sim\vec\theta_2$ when $\vec\theta_1\in \overline A_{\e,1}$, $\vec\theta_2\in \overline A_{\e,2}$,
and $\vec\theta_2 = I_\e(\vec\theta_1)$.
This produces the connected sum manifold $\Mbar_\e$, defined as follows:

\begin{definition}\label{def-Me}
For $i=1,2$ and $a>0$, let $\Bbar_{a,i}$ be the closed subset of $\Mbar$ that corresponds in coordinates to the ball $|\vec{\theta}_i|\le a$. We define $\overline \X_\e\subset\Mbar$ to be the open subset
\begin{equation*}
\overline \X_\e  := \Mbar\smallsetminus
\bigl(\Bbar_{\e^2,1}\cup \Bbar_{\e^2,2}\bigr),
\end{equation*}
and define the spliced manifold
$\Mbar_\e$ by
\begin{equation*}
\Mbar_\e:=\overline \X_\e/{\sim}.
\end{equation*}
We let $\X_\e=\overline \X_\e\cap M$, and let $M_\e$ denote the
subset of $\Mbar_\e$ consisting of points whose representatives are in $\X_\e$.
Let  $\pi_\e\colon \overline \X_\e\to \Mbar_\e$ be the natural quotient map, and define the \emph{neck} of $\Mbar_\e$ to be the
open subset
\begin{equation*}
\overline\n_\e:=\pi_\e \left( \overline A_{\e,1} \right) = \pi_\e \left( \overline A_{\e,2} \right).
\end{equation*}
We let $\n_\e$ denote $\overline\n_\e\cap M_\e$.
\end{definition}

We will parametrize the neck by an expanding family of half-annuli
in the upper half-space.  Let $\Hbar^3 $ denote the closed upper half-space, defined by
\begin{equation*}
\Hbar^3  := \left\{ (y,x^1,x^2)\in \R^3 :  y\ge 0\right\},
\end{equation*}
and let $\h^3 \subset\Hbar^3 $ be the subset where $y>0$.
Let $r  = \bigl( y^2 + (x^1)^2+ (x^2)^2\bigr){}^{1/2}$, and
define $\overline\A_\e \subset \Hbar^3 $ to be the half-annulus
defined by
\begin{equation*}
\overline\A_\e := \left\{ (y,x^1,x^2)\in \Hbar^3 :  \e <r< \frac{1}{\e}\right\},
\end{equation*}
and let $\A_\e = \overline \A_\e \cap \h^3 $.
Analogously to the case of background coordinates, on $\h^3 $ we  use the notations
$x^0=y$ and  $\vec x = (y,x) = (y,x^1,x^2) = (x^0,x^1,x^2)$.

To define the parametrization of the neck,
we first define diffeomorphisms $\alpha_{\e,i}\colon \A_\e \to A_{\e,i}$ for
$i=1,2$ by $\vec\theta_i = \alpha_{\e,i}(y,x) $, where
\begin{align*}
\alpha_{\e,i}(y,x) &= (\e y,\e x).
\end{align*}
Then we define $\beta_\e\colon \A_\e \to A_{\e,2}$ by
\begin{equation*}
\beta_\e  = I_\e \circ \alpha_{\e,1} = \alpha_{\e,2}\circ I,
\end{equation*}
where $I\colon \A_\e\to \A_\e$ is the inversion in the unit circle:
$I(y,x) = (y/r^2,x/r^2)$.
Our preferred parametrization of $\n_\e$ is
\begin{equation*}
\Psi_\e:= \pi_\e \circ \alpha_{\e,1} = \pi_\e\circ\beta_\e = \pi_\e \circ \alpha_{\e,2}\circ I
\colon \A_\e \to \n_\e.
\end{equation*}
The various diffeomorphisms are summarized in the following commutative diagram:
\begin{equation*}
\begin{diagram}[width=.25in]
\A_\e && \rTo^I && \A_\e\\
\dTo<{\alpha_{\e,1}} & \rdTo(4,2)^{\beta_\e} &&& \dTo>{\alpha_{\e,2}}\\
A_{\e,1} && \rTo^{I_\e} && A_{\e,2}\\
& \rdTo & \pi_\e & \ldTo \\
&& \n_\e.
\end{diagram}
\end{equation*}

The topology of $M_\e$ does not change with (sufficiently small) $\e$. The Riemannian geometry on $M_\e$ (which we define next) does depend on $\e$; this is one of  the reasons that we keep track of the parameter $\e$.

To obtain a suitable Riemannian metric $g_\e$ on $M_\e$, we blend the metrics coming from the original annuli
with the use of a cutoff function.

\begin{lemma}\label{cutoff1}
There exists a nonnegative and monotonically increasing smooth cutoff
function $\phi\colon \R\to\R$ that is identically $1$ on $[2,\infty)$,
is supported in $\bigl(\frac{1}{2},\infty\bigr)$, and
satisfies the condition
\begin{equation}\label{phi-eqn}
\phi(r)+\phi \left({\frac{1}{r}}\right)\equiv 1.
\end{equation}
\end{lemma}

\begin{proof}
Let $\phi_0$
be a nonnegative and decreasing
smooth cutoff function  such that $\phi_0(r)=\frac{1}{2}$ for $r\le\frac{1}{2}$ and $\phi_0(r)=0$ for $r\ge2$, and set
\begin{equation*}
\phi(r):=\frac{1}{2}-\phi_0(r)+\phi_0\left({\frac{1}{r}}\right).
\end{equation*}
An easy computation shows that $\phi$ satisfies the conclusions of the lemma.
\end{proof}

Using this cutoff function and the maps $\alpha_{\e,1}$, $\beta_\e$, and $\Psi_\e$  defined above,
we define the metric on $M_\e$ as follows.

\begin{definition}\label{def-ge}
We define $g_\e$ to be the metric on $M_\e$ that agrees with $(\pi_\e)_*g$ away from the neck $\n_\e$, while on $\n_\e$ it satisfies
\begin{equation}\label{def-g}
\Psi_\e^*g_\e =\p(r) (\alpha_{\e,1})^*g +
\p\biggl(\frac{1}{r}\biggr)
(\beta_\e)^* g.
\end{equation}
\end{definition}

Note the following:
\begin{itemize}
\item
On the set where $r\ge 2$,  $\Psi_{\e}^*g_\e$ agrees with $\alpha_{\e,1}^*g$;
\item
On the set where $r\le \tfrac 1 2$, $\Psi_{\e}^*g_\e$ agrees with  $\beta_{\e}^*g$.
\end{itemize}
It is obvious from the definition that $g_\e$ is polyhomogeneous and
conformally compact  of class $C^2$.

\subsection {Splicing the defining functions}\label{defn-func-glue}
Next we construct a family of defining functions for the manifolds $M_\e$ that are
specially adapted to the metrics $g_\e$, and that agree with the original defining function $\rho$ away from $\n_\e$.
To define them, we need the following auxiliary function.

\begin{lemma}\label{lemma:F}
There exist a constant $\b\ge 2$
and a $C^\infty$ function $F\colon(0,\infty)\to (0,\infty)$ that satisfies
\begin{align}
\label{F}
F\left({\frac{1}{r}}\right)&=r^2F(r), && r\in (0,\infty),\\
\label{F=1}
F(r)&=1, &&r\ge \b,\\
\label{F(1/r)}
F(r)&=\frac{1}{r^2}, &&r\le 1/\b,\\
\label{DeltaF}
\frac{|3rF'(r)+r^2F''(r)|}{F(r)}&\le \frac{1}{2}, &&r\in(0,+\infty).
\end{align}
\end{lemma}

\begin{proof}
We use the function 
$1+1/r^2$ to interpolate between $1/r^2$ for small $r$ and 
the constant function $1$ for large $r$.
More specifically, let $\psi_0$ be any smooth cut-off function such that
$$\psi_0(r)= 1\text{ for }r\le 1,
\quad
0\le \psi_0(r)\le 1\text{ for } 1\le r \le 2,
\quad
\psi_0(r)=0 \text{ for } r\ge 2,
$$
let 
$$C = \sup _{r\in (0,\infty)}|3r\psi_0'(r)+r^2\psi_0''(r)|,$$
and choose $\b\ge 2$ large enough that 
\begin{equation}\label{eq:b^2}
\left(\frac{2}{\b}\right)^2\le\frac{1}{2C}.
\end{equation}
The function $\psi(r):=\psi_0(\b r)$ 
satisfies $\psi(r)=1$ for $r\le 1/\b$, $\psi(r)=0$ for $r\ge 2/\b$ and
\begin{equation}\label{psi:estimate}
|3r\psi'(r)+r^2\psi''(r)|\le C\ \ \text{for all}\ r.
\end{equation}
Define
$$F(r):=1+\frac{1}{r^2}-\psi(r)-\frac{1}{r^2}\psi\left(\frac{1}{r}\right)$$
and compute
$$F\left({\frac{1}{r}}\right)=1+r^2-\psi\left(\frac{1}{r}\right)-r^2\psi(r)=
r^2F(r).$$
One readily verifies \eqref{F=1}, \eqref{F(1/r)} and 
\begin{equation}\label{eq:F-cases}
 F(r)=
\begin{cases}
1+r^{-2}-\psi(r) & \text{if\ } r\in\left[1/\b, 2/\b\right]\\
1+r^{-2} &\text{if\ } r\in\left[2/\b, \b/2\right]\\
1+r^{-2}-\tfrac{1}{r^2}\psi\left(\tfrac{1}{r}\right) &\text{if\ } r\in\left[\b/2,\b\right].
\end{cases}
\end{equation}
Note that 
\begin{equation}\label{lowerbounds:F}
F(r)\ge \frac{1}{r^2} \text{\ \ \ if\ \ \ }r\le 1,\ \ \  F(r)\ge 1 \text{\ \ \ if\ \ \ }r\ge 1.
\end{equation}
Furthermore, we see that $3rF'(r)+r^2F''(r)$  is nonzero 
only on the intervals $[1/\b, 2/\b]$ and $[\b/2,\b]$. A straightforward computation using 
\eqref{eq:F-cases}
and \eqref{lowerbounds:F} shows
$$\frac{|3rF'(r)+r^2F''(r)|}{F(r)}\le
\begin{cases}
r^2|3r\psi'(r)+r^2\psi''(r)| & \text{if\ } r\in[1/\b, 2/\b]\\
\quad &\\ 
r^{-2}\left|3\frac{1}{r}\psi'\left(1/r\right)+\frac{1}{r^2}\psi''\left(1/r\right)\right| & \text{if\ } r\in[\b/2,\b].
\end{cases}$$
Property \eqref{DeltaF} is now immediate from  \eqref{eq:b^2} and  \eqref{psi:estimate}.
\end{proof}




Note that the function $F$ of the preceding lemma is bounded below by $1$,
and satisfies an estimate of the form
\begin{equation}\label{F'/F}
\frac{|F'(r)|}{F(r)} \le \frac{C}{r}.
\end{equation}

We now let $\xi\colon \h^3 \to \R$ be the positive function defined by
\begin{equation*}
\xi(y,x) = y F(r),
\end{equation*}
where $F$ is the function of the preceding lemma and $r=|\vec x|$.
A computation using \eqref{F} shows that $\xi \circ I= \xi$, and $\e\xi$ agrees with
$\e y = (\alpha_{\e,1})^*\rho$ where
$r\ge 2$
and with $\e y\circ I = (\beta_\e)^*\rho$ where
$r\le 1/2$.
Define
$\rho_\e\colon M_\e\to \R$ to be the defining function
that is equal to $\rho$ away from the neck,
and on the neck satisfies
\begin{equation*}
\Psi_\e^*\rho_\e = \e\xi = \e y F.
\end{equation*}

We wish to show that $g_\e$ is asymptotically hyperbolic.
To do so, we first need to
find an expression for $g_\e$ along the lines of \eqref{grahamlee}.
Recall that $\ghyp:=y^{-2}\bigl(dy^2+(dx^{1})^2+(dx^{2})^2\bigr)$
is the metric of the upper half-space model of hyperbolic space.
The inversion $I(y,x) = (y/r^2, x/r^2)$ is an isometry for this model:
$I^* \ghyp = \ghyp$.
For $i=1,2$, we can write
\begin{equation}\label{alpha*-g}
(\alpha_{\e,i})^*g
= \ghyp + m_{ab,i}(\e y,\e x)\frac{dx^a}{y}\,\frac{dx^b}{y},
\end{equation}
where $m_{ab,i}$ are the error terms from the expression \eqref{grahamlee}
for $g$ in preferred background coordinates centered at $p_i$.
This metric is uniformly equivalent on $\A_\e$ to $\ghyp$; and because
$m_{ab,i}\in C^{2,\alpha}(\U\cap \Mbar)$ and
$m_{ab,i}(0,0)=0$, it is immediate that as $\e\to 0$, the $C^{2,\alpha}$ norm of the functions $m_{ab,i}(\e y,\e x)$
is $O(\e)$ on any subset of $\A_\e$ where $r$ is bounded above,
including in particular the set where $\phi\ne 1$.
(The $C^{2,\alpha}$ norm in use here is the ordinary Euclidean one inherited from $\R^3$.)

To analyze $(\beta_\e)^*g= I^*(\alpha_{\e,2})^*g$, we compute,
for $a=0,1,2$,
\begin{equation*}
I^* \left( \frac{dx^a}{y} \right)
= \sum_c Q^{ac}(y,x)\, \frac{dx^c}{y},
\quad
\text{where }
Q^{ac}(y,x) = \left(\frac{r^2\delta^{ac} - 2 x^a x^c}{r^2} \right),
\end{equation*}
and therefore
\begin{equation}\label{eq:beta*-g}
(\beta_{\e})^*g
= I^*(\alpha_{\e,2})^*g
=  \ghyp + \sum_{a,b,c,d} m_{ab,2}\left(\frac{\e y}{r^2},\frac{\e x}{r^2}\right)Q^{ac}(y,x)Q^{bd}(y,x)\frac{dx^c}{y}\,\frac{dx^d}{y}.
\end{equation}
Because $Q^{ac}$ is a rational function
with nonvanishing denominator, and
is homogeneous of degree zero, it is
uniformly bounded on $\A_\e$, and thus $(\beta_\e)^*g$ is uniformly equivalent on $\A_\e$
to $\ghyp$.  Moreover, all of the derivatives of $Q^{ac}$
are uniformly bounded on any subset of $\A_\e$ where $r$ is bounded below by a positive constant,
such as the set where $\phi\ne 0$.
Also, on any such subset, an easy argument
shows that the $C^{2,\alpha}$-norm  of
$m_{ab,2}(\e y/r^2,\e x/r^2)$ is $O(\e)$ as
$\e\to 0$.
Combining these observations with
formula \eqref{alpha*-g} for $(\alpha_{\e,2})^*g$, we conclude that
the pullback of $g_\e$ to $\A_\e$ has the form
\begin{equation}\label{metric}
(\Psi_\e)^* g_\e
= \ghyp + k_{ab,\e}(y,x)\frac{dx^a}{y}\,\frac{dx^b}{y},
\end{equation}
where for each $\e$, the function $k_{ab,\e}$ is bounded and in $C^{2,\alpha}(\A_\e)$,
$(\Psi_\e)^* g_\e$ is uniformly equivalent to $\ghyp$ independently of $\e$,
and
\begin{equation}\label{est:kab}
\|k_{ab,\e}\|_{C^{2,\alpha}(\A_c)}=O(\e)
\end{equation}
as $\e \to 0$ for any fixed $c\in(0,1)$.

To obtain estimates for the behavior of $k_{ab,\e}$ at the ideal boundary analogous to
\eqref{ma0-est} and \eqref{mjk-est},
we need to explicitly expand the various terms in \eqref{eq:beta*-g}.
Note that in addition to
the functions $Q^{ab}$ being uniformly bounded, we have
\begin{equation*}
Q^{j0}(y,x)=Q^{0j}(y,x) = \frac{2 y x^j }{r^2} = O\left(\frac{y}{r}\right).
\end{equation*}
Therefore, using \eqref{ma0-est}--\eqref{mjk-est} for $m_{ab,i}$, we have
\begin{equation}\label{k00:computation}
\begin{aligned}
k_{00,\e}(y,x)
&= \p(r)  m_{00,1}(\e y,\e x) + \p\Bigl(\frac{1}{r}\Bigr)\Bigl(
m_{jk,2}\left(\frac{\e y}{r^2},\frac{\e x}{r^2}\right)Q^{0j}(y,x)Q^{0k}(y,x)\\
&\quad + 2 m_{j0,2}\left(\frac{\e y}{r^2},\frac{\e x}{r^2}\right)Q^{0j}(y,x)Q^{00}(y,x)
+ m_{00,2}\left(\frac{\e y}{r^2},\frac{\e x}{r^2}\right)Q^{00}(y,x)Q^{00}(y,x)\Bigr)\\
&= \p(r) O(\e^2 y^2) + \p\Bigl(\frac{1}{r}\Bigr) O\left(\frac{\e y^2}{r^4}\right),
\end{aligned}
\end{equation}
where the implied constants on the right are
uniform in $\e$ on all of $\A_\e$.
Note that, in view of the definition of $\rho_{\e}$, the right-hand side of this equation is
$O(\rho_\e)$, uniformly in $\e$.
Applying similar computations to the other terms, and using the
notation $g_{ab,\e}$ to denote the components of $\Psi_\e^*g_\e$ in standard
coordinates on $\A_\e$,
we conclude that
\begin{align}
\label{g_00}
g_{00,\e}(y,x) &= y^{-2} (1 + O(\rho_\e));\\
\label{g_0j}
g_{0j,\e}(y,x) &= y^{-2} O(\rho_\e), &j&\in \{1,2\};\\
\label{g_jk}
g_{jk,\e}(y,x) &= y^{-2} (\delta_{jk} + O(\e)), &j,k&\in \{1,2\}.
\end{align}
It follows that the inverse matrix satisfies
\begin{align}
\label{g^00}
g_\e^{00}(y,x) &= y^2(1+O(\rho_\e));\\
\label{g^0j}
g_\e^{0j}(y,x) &= y^2 O(\rho_\e), &j&\in \{1,2\};\\
\label{g^jk}
g_\e^{jk}(y,x) &= y^2 (\delta^{jk} + O(\e)), &j,k&\in \{1,2\}.
\end{align}

\begin{lemma}\label{defnfnc}
There exists a constant $C>0$ independent of $\e$ such that  sufficiently close to the ideal boundary $\partial \Mbar_\e$ of $M_\e$ we have
\begin{equation*}
\left|\frac{|d\rho_\e|_{g_\e}^2}{\rho_{\e}^{2}}-1\right|\le \frac {C}{\e}\rho_{\e}.
\end{equation*}
\end{lemma}

\begin{proof}
Away from the neck, $\rho_\e$ and $g_\e$ match the original $\rho$ and $g$ on $M$, and the
result follows there from  Lemma \ref{def-fcn}.
We compute on the neck by identifying it
with $\A_\e$ by means of the diffeomorphism $\Psi_\e$.
We obtain
\begin{equation}\label{drhoe}
\begin{aligned}
\frac{|d\rho_\e|_{g_\e}^2}{\rho_{\e}^{2}}
&= \frac{ \left|\e F(r)dy +\e y F'(r)\,dr\right|_{g_\e}^2}{(\e y F(r))^2}\\
&= \left| \frac{dy}{y}\right|_{g_\e}^2 +2 \left\langle  \frac{dy}{y}, \frac{F'(r)}{F(r)}\,dr\right\rangle_{g_\e}
+ \left|\frac{F'(r)}{F(r)}\,dr\right|^2_{g_\e}.
\end{aligned}
\end{equation}
Using
\eqref{g^00},
we see that
the first term on the right-hand side of
\eqref{drhoe} is
\begin{equation*}
\left| \frac{dy}{y}\right|_{g_\e}^2  = 1 + O(\rho_\e).
\end{equation*}

To estimate the other terms, note that
$\rho_\e$ and $g_\e$
agree on the neck with (pullbacks of) $\rho$ and $g$
except on the subset $\A_{1/2}\subset\A_\e$.
So in the computation below,
we may assume that $1/2\le r\le 2$, which means that $\rho_\e$ is bounded above and
below by constant multiples of $\e y$, and
$\Psi_\e^*g_\e$ is uniformly equivalent
to $\ghyp$.
Thus, up to a constant multiple,
the second term in \eqref{drhoe} is bounded on $\A_{1/2}$ by
\begin{equation*}
\left| \left\langle  \frac{dy}{y}, \frac{F'(r)}{F(r)}\,dr\right\rangle_{\ghyp}\right|
= \frac{|F'(r)|}{F(r)} \, \frac{y^2}{r^2}
\le \frac {C}{\e}\rho_{\e},
\end{equation*}
and the third by
\begin{equation*}
\left|\frac{F'(r)}{F(r)}\,dr\right|^2_{\ghyp}
= \frac{F'(r)^2}{F(r)^2}\,y^2
\le \frac {C}{\e}\rho_{\e}.
\end{equation*}
The result follows.
\end{proof}

The previous lemma implies that $|d\rho_\e|_{\rho_\e^2{g}_\e}= 1$ on the ideal boundary $\partial \Mbar_\e$. Thus,
the manifold $(M_\e,g_\e)$ is asymptotically hyperbolic.  A consequence of this property is that for each value of $\e$, $R(g_\e)$ approaches $-6$ as $\rho_\e\to 0$. 
We need to show that the convergence is uniform in $\e$.

It follows from \cite{GL} that any AH metric $g$ satisfies  
\begin{equation}\label{scalarcurv:comp}
R(g)=-6\left|d\rho\right|^2_{\bar{g}}+\rho\cdot P_1(\bar g, \bar g^{-1}, \partial\bar{g})+\rho^2 \cdot P_2(\bar g, \bar g^{-1}, \partial\bar{g}, \partial^2\bar{g}),
\end{equation}
where (in accord with the definition of asymptotic hyperbolicity)  $\bar{g}=\rho^2 g$ and where $P_m$, $m=1,2$, are certain universal polynomials whose terms involve 
$m^{th}$-order derivatives of the components of $\bar{g}$.  Since (by Lemma \ref{def-fcn})  we have  that $\left|d\rho\right|^2_{\bar{g}}=1+O(\rho^2)$, it follows that
$R(g)+6=O(\rho)$. Clearly then, for each individual $\e$ the function  $\rho_\e^{-1}\left(R(g_\e)+6\right)$ is bounded on $M_\e$. The uniformity question now is whether  the $C^{0,\alpha}$-norms of the functions are bounded uniformly in $\e$.

\begin{lemma}\label{scalcurv:est1}
If $c>0$ is fixed, then
$
\left\|\left(\rho_\e^{-1}\left(R(g_\e)+6\right)\right)\circ \Psi_\e\right\|_{C^{0,\alpha}(\A_{c})}$
is bounded uniformly in $\e$. 
\end{lemma}

\begin{proof} First observe that 
$$\left(\rho_\e^{-1}\left(R(g_\e)+6\right)\right)\circ \Psi_\e=(\e  y F)^{-1} \left(R(\Psi_\e^*g_\e)+6\right).$$
To estimate the scalar curvature $R(\Psi_\e^*g_\e)$ let 
$\bar h_\e:=y^2\Psi_\e^*g_\e$. We  then have 
\begin{equation}\label{scalarcurv:est2}
R\left(\Psi_\e^*g_\e\right)+6=\left(6-6\left|dy\right|^2_{\bar h_\e}\right)+y\cdot P_1(\bar{h}_\e, \bar{h}_\e^{-1}, \partial\bar{h}_\e)+y^2\cdot P_2(\bar{h}_\e, \bar{h}_\e^{-1}, \partial\bar{h}_\e, \partial^2\bar{h}_\e).
\end{equation}
The components of the metric $\bar h_\e$ can be expressed as
\begin{equation*}
\bar h_{ab,\e}(y,x^1, x^2)=\delta_{ab}+k_{ab,\e}(y,x)
\end{equation*}
where  $\|k_{ab,\e}\|_{C^{2,\alpha}(\A_{c})}=O(\e)$ as $\e\to 0$. 
Consequently, the components of $\bar h_\e$, $\bar h_\e^{-1}$, $\e^{-1} \partial \bar h_\e$ 
and  $\e^{-1} \partial^2 \bar h_\e$ have uniformly bounded $C^{0,\alpha}(\A_{c})$-norms. 
It follows that the terms in
$(\e y F)^{-1}(R(\Psi_\e^*g_\e)+6)$ 
coming from the
last two terms of the right-hand side of
\eqref{scalarcurv:est2}
 have uniformly bounded  $C^{0,\alpha}(\A_{c})$-norms. A careful consideration of the expansion \eqref{k00:computation} yields
$$k_{00,\e}(y,x)=O(\e y^2) \text{\ \ and\ \ } \partial k_{00,\e}(y,x)=O(\e y)$$
on $\A_{c}$. Therefore,
$$\left\|\e^{-1}y^{-1}\left(6-6|dy|^2_{\bar h_\e}\right)\right\|_{C^{0,\alpha}(\A_{c})}=
\left\|\e^{-1}y^{-1}k_{00,\e}(y, x)\right\|_{C^{0,\alpha}(\A_{c})}$$
is also uniformly bounded. This observation completes our proof.
\end{proof}

We conclude the discussion of the AH geometries $(M_\e, g_\e)$ and their defining functions $\rho_\e$ by proving that 
each $\rho_\e$ is superharmonic.
\begin{lemma}\label{LaplaceRho}
If $\e>0$ is small enough then $\Delta_{g_\e}\rho_\e\le 0$.
\end{lemma}

\begin{proof}
Away from the gluing region $\Psi_\e\left(\A_{1/\b}\right)$, 
the quotient map $\pi_\e$ is a diffeomorphism 
satisfying $\pi_\e^*g_\e=g$ and $\pi_\e^*\rho_\e=\rho$. 
Thus, away from the gluing region the inequality we need to show is an immediate consequence of \eqref{delta-rho}. 

To prove the inequality on the gluing region we utilize the transformation law for the conformal Laplacian 
(see, e.g., \cite[eq.\ (2.7)]{LP})
to write the Laplace operator for $(\Psi_\e)^*g_\e$ in terms of that of 
the conformally related metric $\bar h_\e:=y^2(\Psi_\e)^*g_\e$: 
$$
\Psi_\e^*\left(\frac{\Delta_{g_\e}\rho_\e}{\rho_\e}\right)=\frac{\Delta_{(\Psi_\e)^*g_\e}(yF)}{yF}=\frac{1}{8}R\left((\Psi_\e)^*g_\e\right)+\frac{y^{3/2}}{F}\left[\Delta_{\bar h_\e}(y^{1/2}F)-\frac{1}{8}R(\bar h_\e)y^{1/2}F\right].
$$
It follows from \eqref{est:kab}  
that 
the difference between $\bar h_\e$ and the Euclidean metric $\delta$ approaches zero, in the sense that $\|\bar h_\e-\delta\|_{C^{2,\alpha}(\A_{1/\b})}\to 0$.
Thus we have $R(\bar h_\e)\to 0$ and
$$y^{3/2}\left[\Delta_{\bar h_\e}(y^{1/2}F)-\Delta_\delta(y^{1/2}F)\right]\to 0 \text{\ \ as\ \ }\e\to 0,$$
with both convergences uniform on $\A_{1/\b}$.
A straightforward computation shows that
$$\Delta_\delta(y^{1/2}F)=-\frac{1}{4}y^{-3/2}F+3\frac{y^{1/2}}{r}F'+y^{1/2}F''.$$
Since Lemma \ref{scalcurv:est1} shows that $R\left((\Psi_\e)^*g_\e\right) = -6 + O\bigl(\Psi_\e^*\rho_\e\bigr)$, which is
equal to $-6+O(\e)$ on $\A_{1/\b}$, it follows that
$$\Psi_\e^*\left(\frac{\Delta_{g_\e}\rho_\e}{\rho_\e}\right)+1-\frac{y^2}{r^2}\cdot \frac{3rF'+r^2F''}{F}\to 0$$
uniformly on $\A_{1/\b}$ as $\e\to 0$.
Our result is now an immediate consequence of \eqref{DeltaF} and the fact that $y^2\le r^2$ everywhere.
\end{proof}

\subsection{Splicing the traceless part of the second fundamental form}
Recall that
our given second fundamental form on $M$ can be written $K=\mu+g$, where
$\mu$ is a traceless, divergence-free symmetric $2$-tensor field
of the form
$\mu = \rho^{-1}\overline\mu$ for some $\overline\mu\in C^{1,\alpha}(\overline M)$.
Our goal in this section is to create  on each $M_\e$ a traceless
symmetric $2$-tensor $\mu_\e$ that
is ``approximately divergence-free,''
and such that $\pi_\e^*\mu_\e$ is equal
to $\mu$ away from the neck.
Later we will
correct it so that it is
divergence-free.

Let $\chi\colon\R\to \R$ be a
smooth nonnegative function such that
$\chi(r)=1$ for $r\ge 3$ and $\supp(\chi) \subset (2,\infty)$.
For each $\e >0$, define $\chi_\e\colon M\to \R$ by
\begin{equation}\label{eq:def-chi}
\begin{aligned}
\chi_\e(\rho,\theta_1^a) &= \chi\left( \frac{\rho^2}{\e^2} + \sum_j\frac{ (\theta_1^j)^2}{\e}\right), &&\text{on $\overline B_{1,1}$},\\
\chi_\e(\rho,\theta_2^a) &= \chi\left( \frac{\rho^2}{\e^2} + \sum_j\frac{ (\theta_2^j)^2}{\e}\right), &&\text{on $\overline B_{1,2}$},\\
\chi_\e &\equiv 1, &&\text{on $M\setminus \bigl(\overline B_{\sqrt{3\e},1}\cup \overline B_{\sqrt{3\e},2}\bigr)$}.
\end{aligned}
\end{equation}
Then let
$\widehat\mu_\e = \chi_\e\mu$ on $M$.
(The level sets of $\chi_\e$ are half-ellipsoids with
radius proportional to $\e$ in the $\rho$ direction  and to $\sqrt{\e}$ in the
$\theta$-directions.  We have designed this unusual cutoff function
so that the divergence of $\widehat\mu_\e$ will be uniformly small,
despite the fact that the tangential and normal components of $\mu$
vanish at different rates near the ideal boundary; see Lemma
\ref{lemma:div-not-too-large} for details.)

It follows from our choice of $\chi_\e$ that $\widehat\mu_\e$ is supported in the set
$M\setminus \bigl(\overline B_{2\e,1}\cup \overline B_{2\e,2}\bigr)$.
Because $\pi_\e$ restricts to a diffeomorphism from this set to an open subset of $M_\e$,
we can define a symmetric $2$-tensor $\mu_\e$ on $M_\e$ by
\begin{equation}\label{eq:def-mue}
\mu_\e = \pi_{\e*}\widehat \mu_\e,
\end{equation}
understood to be zero on the neck.
Because $g_\e = \pi_{\e*}g$ on the support of $\mu_\e$, and $\mu_\e$ is a scalar multiple of
$\pi_{\e*}\mu$, it follows that $\mu_\e$ is traceless with respect to $g_\e$.
Although it is generaly not divergence-free, we will show below that its divergence is not too large
(see Lemma \ref{lemma:div-not-too-large}).

\section{Analysis on the Spliced Manifolds}\label{sec:analysis}

In this section we develop the results we will need
about linear elliptic operators on our spliced manifolds.

\subsection {Weighted H\"older spaces and linear differential operators}\label{holder}

To carry out the needed analysis on AH geometries with AH data, and also to provide a convenient framework
for specifying the rate at which various quantities like the trace free part of $K$
approach their requisite asymptotic values, it is convenient to work with weighted H\"older spaces.
Here we recall the definition of these spaces, using the conventions of \cite{lee}.

Suppose $(M,g)$ is an asymptotically hyperbolic Riemannian geometry of class $C^{\ell,\beta}$ and $\rho$ is a
smooth defining function.
(Our polyhomogeneous metrics, for example, are automatically asymptotically hyperbolic of
class $C^{2,\alpha}$ for every $\alpha\in (0,1)$.)
Let $(\rho,\theta^1,\theta^2)$ be background coordinates on an open subset $\U\subset\Mbar$, which we may assume
extend to a neighborhood of the closure of $\U$ in $\Mbar$.
Let $\breve B\subset \h^3 $ be a fixed precompact ball containing $(1,0,0)$.
A \emph{M\"obius chart for $M$} (or more accurately a \emph{M\"obius parametrization})
is a diffeomorphism $\Phi\colon \breve B\to \U$ whose
coordinate representation has the form
\begin{equation*}
(\rho,\theta^1,\theta^2) =
\Phi(y,x) = (a y, a x^1 + b^1, a x^2 + b^2)
\end{equation*}
for some constants $(a,b^1,b^2)$.
There is a neighborhood $W$ of $\partial M$ in $\Mbar$ covered by finitely many
background charts, and then the resulting family of M\"obius charts covers
$W\cap M$.  We extend this cover to all of $M$ by choosing finitely many interior charts,
which we  also call M\"obius charts for uniformity, to cover $M\smallsetminus W$.

Let $E$ be a tensor bundle over $M$.  For any nonnegative integer $k$ and  real number $\alpha\in [0,1]$ such that
$k+\alpha\le \ell+\beta$,
we define the intrinsic H\"older space $C^{k,\alpha}(M;E)$  as the set of sections $u$ of $E$ whose coefficients
are locally of class $C^{k,\alpha}$, and for which the following norm is finite:
\begin{equation*}
\|u\|_{k,\alpha} = \sup_{\Phi} \| \Phi^* u\|_{C^{k,\alpha}(\breve B)},
\end{equation*}
where the supremum is over our collection of M\"obius charts,
and the norm on the right-hand side is the usual Euclidean H\"older norm of the components of a tensor field
on $\breve B$.
For any real number $\delta$, we define the corresponding \emph{weighted H\"older space} by
\begin{equation*}
C^{k,\alpha}_\delta(M;E) = \{\rho^\delta u: u\in C^{k,\alpha}(M;E)\},
\end{equation*}
with norm
\begin{equation*}
\|u\|_{k,\alpha,\delta} = \|\rho^{-\delta}u\|_{k,\alpha}.
\end{equation*}
When the tensor
bundle is clear from the context, we will usually abbreviate
the notation by writing
$C^{k,\alpha}_\delta(M)$ instead of
$C^{k,\alpha}_\delta(M;E)$.
The index $\delta$ labels the rate of asymptotic decay of a given quantity,
measured in terms of the intrinsic (asymptotically hyperbolic) Riemannian metric $g$.
In particular, we note that larger positive values of $\delta$ imply more rapid decay.
It is shown in \cite[Lemma 3.7]{lee} that if $\eta$ is any covariant $r$-tensor field on $\overline M$
with coefficients in background coordinates that are
$C^{k,\alpha}$ up to the ideal boundary, then $\eta\in C^{k,\alpha}_r(M)$.  Similarly,
any vector field with coefficients that are $C^{k,\alpha}$ up to the ideal boundary
lies in $C^{k,\alpha}_{-1}(M)$.

A linear partial differential operator $P$ of order $m$ between tensor bundles is said to be
\emph{geometric} if the components of $Pu$ in any coordinates can be expressed as
linear functions of the components of $u$ and their covariant derivatives of order
at most $m$, with coefficients that are constant-coefficient polynomials
in the dimension, the components of $g$,
their partial derivatives,
and $1/\sqrt{\det g_{ij}}$,
such that the coefficient of each $j$th derivative
of $u$ involves at most the first $m-j$ derivatives of $g$.  The operators
$\Delta_g$ (the Laplace-Beltrami operator),
$\mathrm{div}_g$ (the divergence),
$\mathcal{D}_{g}$ (the conformal Killing operator),
$\mathcal{D}_{g}^*$ (the adjoint of $\mathcal{D}_{g}$), and
$\vlap_{g}$ (the vector Laplacian) introduced above are all examples of geometric operators.
It is shown by Mazzeo \cite{Mazzeo-edge} (see also \cite{lee})
that every geometric operator $P$ of order $m$ on an asymptotically hyperbolic Riemannian geometry of class $C^{l,\beta}$
defines a bounded linear map from $C^{k,\alpha}_\delta(M)$ to
$C^{k-m,\alpha}_\delta(M)$:
\begin{equation}\label{Holder-bounded}
\|Pu\|_{k-m,\alpha,\delta} \le C\|u\|_{k,\alpha,\delta},
\end{equation}
whenever $m\le k+\alpha\le \ell+\beta$; and
moreover, if $P$ is also
\emph{elliptic}, then it satisfies the following elliptic estimate for $0<\alpha<1$
and  $m< k+\alpha\le \ell+\beta$:
\begin{equation}\label{Holder-elliptic}
\|u\|_{k,\alpha,\delta}\le C \bigl( \|Pu\|_{k-m,\alpha,\delta} + \|u\|_{0,0,\delta}\bigr).
\end{equation}

Because our spliced manifolds $(M_\e,g_\e)$
are polyhomogeneous and asymptotically hyperbolic of class $C^2$,
they are also asymptotically hyperbolic of class $C^{2,\alpha}$ for some $\alpha\in(0,1)$,
and thus
the results we have just discussed hold on $M_\e$ for
each $\e$,
with $(\ell,\beta)=(2,\alpha)$.
However, for our subsequent analysis, we need to check that the constants in
\eqref{Holder-bounded} and
\eqref{Holder-elliptic} can be chosen \emph{independently of $\e$} when $\e$ is
sufficiently small.  Threading through
the arguments of \cite{lee}, we see that for a given geometric operator $P$,
the constants depend only on
uniform bounds of the following type
as $\Phi$ ranges over a collection
of M\"obius charts covering $M$:
\begin{equation}\label{unif-metric-est}
\|\Phi^*g - \ghyp\|_{C^{2,\alpha}(\breve B)}\le C,
\qquad
\sup _{\breve B} \left| \left( \Phi^* g\right)^{-1}\ghyp\right|\le C.
\end{equation}
(In \cite{lee}, attention is restricted to a countable, uniformly locally
finite family of M\"obius charts, but that additional restriction is used
only for Sobolev estimates, which do not concern us here.)
In fact, it is not necessary to use M\"obius charts per se,
in which the first background coordinate is exactly equal to $\rho$;
the arguments of \cite{lee} show that it is sufficient to
use any family of parametrizations $\Phi$ satisfying \eqref{unif-metric-est},
as long as there is a precompact subset $\breve B_0$ of $\breve B$ such that
the images of the restrictions $\Phi|_{\breve B_0}$ still cover $M$,
and the
following uniform
estimates hold in addition to  \eqref{unif-metric-est}:
\begin{equation}\label{unif-rho-est}
\|\Phi^*\rho\|_{C^{2,\alpha}(\breve B)} \le C\rho_0,
\qquad
\frac{1}{C} \rho_0 \le
\left| \Phi^*\rho\right|
\le C\rho_0,
\end{equation}
where $\rho_0=\Phi^*\rho(1,0,0)$.
Thus to obtain our uniform estimates, we need only exhibit a family of charts
for each $\e$ such that
the corresponding
estimates hold for $g=g_\e$ and $\rho=\rho_\e$,
with constants
independent of $\e$.

Start with the family of all M\"obius charts for $M$.  On the portion of $M_\e$ away
from the neck, these same charts (composed with $\pi_\e$; see Definition \ref{def-Me})  serve as
charts for $M_\e$, which  satisfy
\eqref{unif-metric-est} and \eqref{unif-rho-est} uniformly in $\e$.
Recall that we use the diffeomorphism
$\Psi_\e\colon \A_\e\to\n_\e$ to parametrize the neck.
Because $\Psi_\e^*g_\e$ is isometric to $g$ except on a subset of $\A_{1/2}\subset\A_\e$,
we need only
show how to construct
appropriate charts covering points in $\Psi_\e(\A_{1/2})$.

On this set, we will use standard coordinates on $\A_{1/2}$ as a substitute for
background coordinates.
Given $p = \Psi_\e(y_0,x^1_0,x^2_0)\in \Psi_\e(\A_{1/2})$, we define
$\Phi_p = \Psi_\e\circ \phi_p\colon \breve B\to \n_\e$,
where $\phi_p\colon \breve B\to \A_\e$ is the map
\begin{equation*}
\phi_p(x,y) = (y_0 y, y_0 x^1 + x^1_0, y_0 x^2 + x^2_0).
\end{equation*}
Note that the Jacobian of $\phi_p$ is $y_0$ times the identity.
Under this map, \eqref{metric} shows that
$g_\e$ pulls back to
\begin{equation*}
\Phi_p^* g_\e
= \ghyp + k_{ab,\e}(y_0 y, y_0 x^1 + x^1_0, y_0 x^2 + x^2_0)\frac{dx^a}{y}\,\frac{dx^b}{y},
\end{equation*}
where we recall that $\ghyp$ is the metric of the upper half space model of hyperbolic space.
If we assume that $\e$ is small enough that $\phi_p(\breve B)\subset \A_{1/4}$,
these metrics satisfy
the estimates in \eqref{unif-metric-est} uniformly in $\e$
because
the functions $k_{ab,\e}$ are uniformly small in $C^{2,\alpha}$ norm on
$\A_{1/4}$.
The defining function
$\rho_\e$ pulls back to
\begin{equation*}
\Phi_p^*\rho_\e(\vec x) = \e y_0 y F\bigl(|\phi_p(\vec x)|\bigr),
\end{equation*}
and $\Phi_p^*\rho_\e(1,0,0) = \e y_0 F\bigl(|(y_0,x^1_0,x^2_0)|\bigr)$.
Because $F$ is uniformly bounded above and below on $\A_{1/4}$
by positive constants, and all of
its derivatives are uniformly bounded there, it follows that the functions
$\Phi_p^*\rho_\e$ satisfy the estimates in \eqref{unif-rho-est}
uniformly in $\e$.

Summarizing the discussion above, we have proved

\begin{proposition}\label{uni-reg}
Suppose $P$ is a geometric operator of order $m\le 2$  acting on
sections of a tensor bundle $E\to M$, and for each $\e>0$,
$P_\e$ is the corresponding
operator on $M_\e$.
There exists a constant $C$ independent of $\e$ such that for all $C^{2,\alpha}$ sections
$u$ of $E$, all integers $k$ such that $m\le k \le 2$,
and all real numbers $\delta$,
\begin{equation*}
\|P_{\e} u\|_{k-m,\alpha,\delta}\le C\|u\|_{k,\alpha,\delta}.
\end{equation*}
If in addition $P$ is elliptic, then
\begin{equation*}
\|u\|_{k,\alpha, \delta}\le C\bigl(\|P_{\e}u\|_{k-m,\alpha,\delta}+\|u\|_{0,0,\delta}\bigr).
\end{equation*}
\end{proposition}

\subsection{The Vector Laplacian on Hyperbolic Space}\label{Laplacian}

In this section, we  study the kernel of the vector Laplacian
$\vlap_{\ghyp}=\mathcal{D}_{\ghyp}^*\circ \mathcal{D}_{\ghyp}$
on
hyperbolic space $(\h^3 ,\ghyp)$.
We denote the standard coordinates by $( y, x) = ( y, x^1, x^2) = (x^0,x^1,x^2)$ on $\h^3 $, and
we use the notations
$| x|=\bigl(( x^{1})^2+( x^{2})^2\bigr){}^{1/2}$
and $r =|\vec x| = \bigl( | x|^2+ y^2\bigr){}^{1/2}$.
As a global defining function on $\h^3 $, we use
\begin{equation*}
\rhohyp( y, x^1, x^2) = \frac{2 y}{| x|^2 + ( y+1)^2}.
\end{equation*}
The function $\rhohyp$ is the pullback to the upper half-space of the
usual defining function
$\tfrac 1 2 (1-|x|^2)$ on the unit ball.

It is well known (see, for example, \cite {IMP1} or  \cite{lee})
that the vector Laplacian
$$\vlap_{\ghyp}:C^{2,\alpha}_\delta(\h^3 )\to C^{0,\alpha}_\delta(\h^3 )$$ is invertible
for $-1<\delta<3$ and $0<\alpha<1$.
This leads to the following lemma.

\begin{lemma}\label{L-injectivity}
If $-1<\delta<3$, then there is no nonzero global vector field $X$ on $\h^3 $ satisfying both $\lhyp X=0$
and the estimate $|X|_{\ghyp}\le C\rhohyp^\delta$.
\end{lemma}

\begin{proof}
The hypothesis implies that $X\in C^{0,0}_\delta(\h^3 )$, and then Lemma 4.8(b)
of \cite{lee} implies that $X\in C^{2,\alpha}_\delta(\h^3 )$ for $0<\alpha<1$.
The result then follows from the injectivity of $\vlap_{\ghyp}$ on the latter space.
\end{proof}

We need some variations on this result, in which the
defining function $\rhohyp$ is replaced by other weight functions.
As in Section \ref{defn-func-glue}, let $\xi\colon \h^3 \to\R$ be the function $\xi( y, x) =  y F(r)$,
where $F$ is the  function of Lemma \ref{lemma:F}.
We noted earlier that $\xi\circ I = \xi$, where
$I\colon \h^3 \to \h^3 $ is the $\ghyp$-isometry given by
inversion with respect to the unit hemisphere.
Away from $0$ and $\infty$, $\xi$ is
a defining function for $\partial\h^3 $, but it blows up at both $0$ and $\infty$.

\begin{proposition}\label{eugene'}
If $-1<\delta<3$, then there is no nonzero global vector field $X$ on $\h^3 $ satisfying both $\lhyp X=0$
and the estimate $|X|_{\ghyp}\le C\xi^\delta$.
\end{proposition}

\begin{proof}
Suppose  $X$ is a nonzero vector field on $\h^3 $ satisfying $\lhyp X=0$ and $|X|_{\ghyp}\le C\xi^\delta$ for some $-1<\delta<3$.
As a consequence of the behavior of the metric $\ghyp$ near the ideal boundary,
the components of $X$ in standard coordinates on $\h^3$ satisfy the condition
$|X^a|\le C y\xi^\delta$.

Let $\varphi\colon \R^2\to \R$ be a smooth bump function
supported in the set where $| x|\le \frac{1}{2}$ and satisfying $0\le \varphi( x)\le 1$,
and define a smooth vector field $Y=(Y^0,Y^1,Y^2)$ on $\h^3 $ by
\begin{equation}\label{eq:def-Y}
Y^a( y, x) = \int_{\R^2} X^a ( y,  x-u)\varphi(u)\, du.
\end{equation}
Differentiation under the integral sign shows that $\lhyp Y=0$.

We  show first that on the set where
$r\le 1$, $Y$ satisfies an estimate of the form
$|Y|_{\ghyp}\le C y^s$ for some $s$ with
$-1<s<3$.
Observe that  by definition $\xi( y, x) = O( y/r^2)$ for $r\le 1$,
and consequently
\begin{equation*}
|X^a( y,  x-u)|\le C y\left(\frac{ y}{| x-u|^2+ y^2}\right)^\delta \text{ for } r=|( y, x)|\le 1.
\end{equation*}
Making the substitution $u =  x -  yv$ (thereby defining $v$), we have
\begin{align*}
|Y^a|
&\le C\int_{|u|\le \frac{1}{2}}  y\left(\frac{ y}{| x-u|^2 +  y^2}\right)^\delta\, du\\
&= C\int_{| x- y v|\le \frac{1}{2}}  y\left(\frac{ y}{ y^2| v|^2 +  y^2}\right)^\delta  y^2\,d v\\
&= C y^{3-\delta}\int_{| x- y v|\le \frac{1}{2}}\frac{d v}{(| v|^2 + 1)^\delta}.
\end{align*}
Because $|( y, x)|\le 1$, the triangle inequality implies that $\{ v: | x- y v|\le \frac{1}{2}\}$
is contained in $\{ v:| v|\le 2/ y\}$.  We now distinguish three cases.

{\sc Case 1:} If $\delta>1$, then the integrand above has finite integral over all
of $\R^2$.  Therefore, $|Y^a|\le C  y^{3-\delta}$, from which it
follows that $|Y|_{\ghyp}\le C  y^{2-\delta}$.

{\sc Case 2:} If $\delta<1$, then we let $(t,\omega)$ denote polar coordinates
in the $( v^1, v^2)$ plane, and we compute
\begin{align*}
|Y^a|
&\le  y^{3-\delta}\int_{0}^{2\pi}\int_{0}^{2/ y}\frac{t\,dt\,d\omega}{(t^2 + 1)^\delta}\\
&= C y^{3-\delta}\left[ (t^2+1)^{-\delta+1}\right]_{t=0}^{t=2/ y}\\
&\le  C' y^{3-\delta}\left(1 +  y^{2\delta-2} \right)\\
&\le C''  y^{1+\delta}.
\end{align*}
It follows that
$|Y|_{\ghyp}\le C''  y^{\delta}$.

{\sc Case 3:} If $\delta=1$, then computing in polar coordinates as before, we get
\begin{align*}
|Y^a|
&= C y^{2}\Bigl[ \log (t^2+1)\Bigr]_{t=0}^{t=2/ y}\\
&\le  C' y^2|\log y|\le C' y^{1+s},
\end{align*}
for any $s$ such that $0<s<1$.  It follows that $|Y|_{\ghyp}\le C  y^{s}$.

In the three cases above,
on the set where $r\le 1$,
we have obtained an estimate of the form $|Y|_{\ghyp}\le C  y^s$
for some $s$
such that $-1<s\le \min \{\delta, 2-\delta\}$.
On the other hand, if $r\ge 1$ and $|u|\le \tfrac{1}{2}$,  then we have
$|( y, x-u)|\sim |( y, x)|$ and $\xi( y, x-u)\sim \xi( y, x)$,
where $\sim$ means ``bounded above and below by constant multiples of.''
It follows easily that $|Y^a|\le C  y\xi^{\delta}$,
and therefore  $|Y|_{\ghyp}\le C \xi^{\delta}$ on this set.

Now let $\widetilde Y$ be the vector field $\widetilde Y = I_*Y$.  Because $I$ is
an isometry and $\xi$ is
$I$-invariant, the argument above implies that
\begin{equation*}
|\widetilde Y|_{\ghyp} \le
\begin{cases}
C \xi^\delta, &r\le 1,\\
C ( y\circ I)^s, & r\ge 1.
\end{cases}
\end{equation*}
Defining a new vector field $Z$ on $\h^3 $ by
\begin{equation*}
Z^a( y, x) = \int_{\R^2} \widetilde Y^a ( y,  x-u)\varphi(u)\, du,
\end{equation*}
we find that  $\lhyp Z=0$, and consequently
the same argument as above shows that
$Z$ satisfies the estimate
\begin{align*}
|Z|_{\ghyp} &\le
\begin{cases}
C  y^s, &r\le 1,\\
C ( y\circ I)^s, & r\ge 1
\end{cases}\\
&\le C' \rhohyp^s.
\end{align*}
As a consequence  of Lemma \ref{L-injectivity}, this implies that $Z\equiv 0$.

If $X\not\equiv 0$, choose a point $( y_0, x_0)\in \h^3 $ at which
some coordinate component $X^a( y_0, x_0)$ is nonzero.  After a translation in the
$ x$-variables (which is an isometry of $\h^3 $), we may assume that $ x_0=0$.
There is some ball $B_r(0)\subset\R^2$ such that $X^a( y_0, x)$ does not
change sign for $ x\in B_r(0)$.  If $\phi$ is chosen to be
supported in this ball, it follows from \eqref{eq:def-Y}
that $Y^a( y_0,0)\ne 0$.  Repeating this argument with $Y$ in place
of $X$ shows that there is a point
at which $Z\ne 0$.  This is a contradiction, so we conclude that
$X\equiv 0$ as claimed.
\end{proof}

We also need
the following consequence of this result, in which
the weight function is taken to be the vertical coordinate $ y$.

\begin{corollary}\label{eugene}
If $-1<\delta<3$, then there is no nonzero global vector field $X$ on $\h^3 $ satisfying both $\lhyp X=0$
and the estimate $|X|_{\ghyp}\le C y^\delta$.
\end{corollary}

\begin{proof}
If $\delta\ge 0$, this follows from the previous proposition and the fact
that $ y\le C\xi$.  If $\delta<0$, then it follows from
Lemma \ref{L-injectivity}
and the fact that $ y\ge C\rhohyp$.
\end{proof}

\subsection{The Vector Laplacian on the Spliced Manifolds}\label{Laplacian-epsilon}

We now consider the vector Laplacians on our spliced manifolds.  For each $\e>0$,
let $(M_\e,g_\e)$
be the asymptotically hyperbolic spliced manifold defined in Definitions
\ref{def-Me} and \ref{def-ge},
and let $\vlap_\e := \vlap_{g_\e}$ be its corresponding vector Laplacian.
Since $g_\e$ is asymptotically hyperbolic of class $C^{2,\alpha}$ for some
$\alpha\in (0,1)$,
the analysis in \cite {IMP1} or  \cite{lee} shows that
$\vlap_\e:C^{2,\alpha}_\delta(M_\e)\to C^{0,\alpha}_\delta(M_\e)$
is invertible
so long as $\e>0$ and $-1<\delta<3$.
We need to show that the norm of its inverse is
bounded uniformly in $\e$. The main goal of this section is to understand this uniformity.

Fix $\alpha$ as above and $\delta\in (-1,3)$.
We start with the uniform Schauder estimate (see
Proposition \ref{uni-reg})
\begin{equation}\label{eq:schauder-vlap}
\|X\|_{2,\alpha,\delta}\le
C\bigl(\|\vlap_\e X\|_{0,\alpha,\delta}
+\|X\|_{0,0,\delta}\bigr),
\end{equation}
where $C$ is some constant independent of $\e$. We will show that there is a uniform constant $D$ such that for sufficiently small $\e$
\begin{equation}\label{main}
\|X\|_{0,0,\delta}\le D \|\vlap_\e X\|_{0,0,\delta}.
\end{equation}
This last estimate implies that
$\|X\|_{2,\alpha,\delta}\le
C(D+1)\|\vlap_\e X\|_{0,\alpha,\delta};$
i.e., that the norm of the inverse $(\vlap_\e)^{-1}$ is bounded above by ${1}/(C(D+1))$.

We  use blow-up analysis to prove (\ref{main}). The main ingredient in the analysis is the following lemma.

\begin{lemma}\label{general}
Let $( \Sigma,\gamma)$ be an asymptotically hyperbolic manifold, and let
$\{N_j\}$ be a sequence of open subsets of $\Sigma$
such that every compact subset of $\Sigma$ is contained
in $N_j$ for all but finitely many $j$.
Suppose that for each $j$ we are given
a Riemannian metric $g_j$ on $N_j$ such that
$g_j\to\gamma$
uniformly with two derivatives on every compact subset of $\Sigma$.
Assume furthermore that
there exist vector fields $Y_j$ on $N_j$,
a positive real-valued function $\zeta$ on $\Sigma$,
a compact subset $K_0\subset\bigcap_j N_j$,
and positive constants $C_1$, $C_2$ such that
\begin{enumerate}\letters
\item\label{Yj-bdd-below}
$\displaystyle \inf_{K_0}\left(\zeta^{-\delta}|Y_j|_{g_j} \right)\ge C_2$;
\item\label{Yj-bdd-above}
$\displaystyle\sup_{N_j}\left(\zeta^{-\delta}|Y_j|_{g_j}\right) \le C_1$;
\item\label{LYj-bdd-above}
$\displaystyle\sup_{N_j}\left(\zeta^{-\delta}|\vlap_{g_j}Y_j|_{g_j}\right) \to 0$.
\end{enumerate}
Then there exist a $C^\infty$ vector field $Y$ on $ \Sigma$ and a constant $C_3$ for which
\begin{equation*}
\vlap_{\gamma}Y=0,\qquad |Y|_{\gamma}\le C_3\zeta^\delta,\qquad Y\not \equiv 0.
\end{equation*}
\end{lemma}

\begin{proof}
Let $K\subset  \Sigma$ be a precompact open set, and let $\widehat K$ be a slightly larger precompact open set
containing $\overline K$.
Since the metrics $g_j$ converge uniformly on $\widehat{K}$
(with two derivatives) to $\gamma$, we may assume that the following estimates hold on $\widehat K$ when $j$ is sufficiently large:
$$|Y_j|_{\gamma}\le 2C_1\zeta^{\delta}, \qquad
|\vlap_{\gamma}Y_j|_{\gamma}\le C\zeta^\delta,$$
for some constant $C$ independent of $j$. The function $\zeta$ is bounded above and below by positive constants
on $\widehat{K}$, so it follows that $\|Y_j\|_{H^{0,p}(\widehat{K},\gamma)}$ and $\|\vlap_{\gamma}Y_j\|_{H^{0,p}(\widehat{K},\gamma)}$ are bounded uniformly in $j$.
Sobolev estimates now imply  that $$\|Y_j\|_{H^{2,p}(K,\gamma)}\le C_K,$$
for some new constant $C_K$ depending on $K$ but independent of $j$.

By the Rellich Lemma there exists a subsequence
$Y_{j_n,K}$ of $Y_j$ that converges in $H^{1,p}(K ,\gamma)$.
For $p> 3$, we have a Sobolev embedding $H^{1,p}(K ,\gamma)\to C^{0,0}(K ,\gamma)$. This means that there exists a pointwise limit
$$Y_{K }:=\lim_{j_n\to \infty}Y_{j_n, K }$$ where $Y_{K }\in C^{0,0}(K ,\gamma)$.
Note that by construction $|Y_{K }|_\gamma\le 2C_1\zeta^\delta$ on $K$.

Consider a nested sequence of precompact open sets whose union is $\Sigma$:
$$K_1 \subset \overline{K}_1\subset K_{2} \subset \overline {K}_2\subset K_{3} \subset \overline {K}_3\dots.$$
We may use the process outlined above to inductively construct sequences $Y_{j_n,K_m }$ for each $K_m$, such that
the sequence $Y_{j_n,K_m }$ is a subsequence of $Y_{j_n, K_{m-1} }$ that converges uniformly on $K_m$.
The diagonal sequence $Y_{j_n,K_n}$ converges uniformly on every compact subset of $\Sigma$
to a continuous limit $Y$ on $\Sigma$ that satisfies $|Y|_\gamma\le 2C_1\zeta^\delta$.
The assumption
\eqref{Yj-bdd-below}
ensures that $Y\not \equiv 0$. So, it remains to show that $\vlap_{\gamma}Y=0$.

Let $X$ be a compactly supported test vector field on $ \Sigma$.
Since $g_{j_n}$ converges to $\gamma$ uniformly on $\supp X$ with two derivatives, we have
\begin{align*}
\int_{\Sigma} \left\<\vlap_{\gamma}X,Y\right\>_\gamma dV_\gamma
&=\lim_{j_n\to \infty} \int_{\Sigma}\left\<\vlap_{g_{j_n}}X,Y_{j_n,K_n }\right>_\gamma dV_\gamma\\
&= \lim_{j_n\to \infty} \int_{\Sigma}\left\<X,\vlap_{g_{j_n}}Y_{j_n,K_n }\right>_\gamma dV_\gamma
=0.
\end{align*}
Thus $Y$ is a weak solution to $\vlap_\gamma Y=0$, and it
follows from elliptic regularity that $Y\in C^\infty( \Sigma)$ and $\vlap_{\gamma}Y=0$.
\end{proof}

We now focus on verifying inequality (\ref{main}).
\begin{lemma} \label{vectLapl:lemma} If $-1<\delta<3$, then there exists a constant $D$ such that for sufficiently small $\e$ and for all
vector fields $X\in C^{2,0}_\delta(M_\e)$ we have
\begin{equation}\label{eq:unif-C0}
\|X\|_{0,0,\delta}\le D \|\vlap_\e X\|_{0,0,\delta}.
\end{equation}
\end{lemma}

\begin{proof}
Suppose not: Then there exist positive numbers $\e_j\to 0$ and vector fields $X_j\in C^{2,0}_\delta(M_{\e_j})$ such that
$$\|X_j\|_{0,0,\delta}=1 \text{\ \ and\ \ }
\|\vlap_{\e_j}X_j\|_{0,0,\delta}\to 0.$$ The fact that $\|X_j\|_{0,0,\delta}=1$ means in particular that for each
$j$ there exists a point $q_{\e_j}\in M_{\e_j}$ such that
$$|X_j(q_{\e_j})|_{g_{\e_j}}\ge\tfrac12 { \rho_{\e_j}(q_{\e_j})^\delta}.$$
Since $\pi_\e$ maps $\Omega_{\e/3} = M\smallsetminus
\bigl(\Bbar_{\e/3,1}\cup \Bbar_{\e/3,2}\bigr)$ surjectively onto $M_\e$,
for each $j$ we may choose a representative $q_j$ for $q_{\e_j}$ such that
$q_j\in \Omega_{\e/3}$.
Passing to a subsequence if necessary,
we may assume that $q_j\to q\in \Mbar$.
Our proof now splits into several cases depending on the location of $q$. Each case  culminates in a contradiction.

{\sc Case 1:} $q\in M$. This is the easiest of the cases as it allows immediate use of Lemma \ref{general}.
Indeed, let $(\Sigma,\gamma) = (M,g)$,   $N_j=\Omega_{\e_j/3}$, $g_j\equiv g$, $\zeta\equiv \rho$, and  let $K_0$ be a compact
set containing a small neighborhood of $q$.
Vector fields $Y_j$ on $N_j$ for which $(\pi_{\e_j})_*Y_j=X_j$ necessarily satisfy the hypotheses of Lemma \ref{general}.
Thus there exists a $C^\infty$ nonzero vector field $Y$ on $M$ with $\vlap_{g}Y=0$ and $|Y|_{g}=O( \rho^\delta)$.
However, for  $\delta\in(-1,3)$, the vector Laplacian has no kernel in $C^{0,0}_\delta(M)$, so this is a contradiction.

{\sc Case 2:}
$q\in \partial\Mbar\smallsetminus \{p_1,p_2\}$. Let $\vec{ \theta}:=( \rho,  \theta^1,  \theta^2)$ be a set of 
preferred background
coordinates centered at $q$,
which we may assume to be defined on a half-disk $\mathbf{D}\subset\Mbar$ whose coordinate
radius is $R$.
Let $( \rho_j, \theta^1_{j}, \theta^2_{j})$
be the coordinates of $q_j$, $j\gg 0$.
Note that $( \rho_j,  \theta^1_{j},  \theta^2_{j})\to (0,0,0)$ as $j\to \infty$.

Let $(\Sigma,\gamma)$ be the hyperbolic space $(\h^3 ,\ghyp)$. We define
$N_j$ to be the half-ball  in $\h^3$ centered at $(0,-{ \theta^1_{j}}/{ \rho_j},-{ \theta^2_{j}}/{ \rho_j})$
of (Euclidean) radius ${R}/{ \rho_j}$.
As soon as $j$ is large enough that $|\vec\theta_j|<R/2$, the triangle inequality shows
that the half-ball of radius $R/2\rho_j$ centered at $(0,0,0)$ is contained in $N_j$, so
we see that $\bigcup_j N_j=\Sigma$
and that each compact subset of $M$ is contained in $N_j$ for all but finitely many $j$.

Consider the transformations
$\mathcal{T}_j:N_j \to M$ whose coordinate representations are given by
\begin{equation}\label{t's}
\mathcal{T}_j( y, x^1, x^2)=( \rho_j y,  \rho_j x^1+ \theta^1_{j},  \rho_j x^2+\theta^2_{j}).
\end{equation}
These transformations are chosen so that $\mathcal{T}_j(1,0,0)=q_j$.
We will now construct metrics and vector fields on $N_j$ satisfying the hypothesis of Lemma \ref{general}.

First, let $g_j:=\mathcal{T}_j^*{g}$.
It follows easily from Lemma
\ref{def-fcn}
that
$g_j\to \breve{g}$ uniformly on compact sets together with two derivatives.

Now define $Y_j:= \rho_j^{-\delta}\mathcal{T}_j^*X_j$. We compute:
\begin{align*}
|Y_j(\vec x)|_{g_j}&= \rho_j^{-\delta}|X_j(\mathcal T_j( \vec x))|_{g}\le  \rho_j^{-\delta}( \rho_j y)^\delta= y^\delta,\\
|Y_j(1,0,0)|_{g_j}&= \rho_j^{-\delta}|X_j(q_j)|_{g}\ge  \tfrac12 \rho_j^{-\delta} \rho_j^\delta=\tfrac{1}{2},\\
y^{-\delta}|\vlap_{g_j}Y_j|_{g_j}&=
( \rho_j y)^{-\delta}\Big(|\vlap_{g}X_j |_{g}\circ \mathcal{T}_j\Big)
\le \|\vlap_{\e_j}X_j\|_{0,0, \delta}\to 0.
\end{align*}
In particular, the hypotheses of Lemma \ref{general} are fulfilled for $K_0=\{(1,0,0)\}$ and $\zeta\equiv  y$.
It follows that there is a nonzero vector field $Y$ on $\h^3 $ for which
$\lhyp Y=0$  and $|Y|_{\ghyp}=O( y^\delta)$. For $\delta\in(-1,3)$ this is impossible by Corollary \ref{eugene}.

{\sc Case 3:} $q\in\{p_1,p_2\}$; without loss of generality we may assume that $q=p_1$.
For sufficiently large $j$, the point $q_{\e_j}$ is contained in the neck $\n_{\e_j}$;
let $\vec{ x}_j:=( y_j, x^1_{j}, x^2_{j})$ be the point in $\A_{\e_j}$ such that $\Psi_{\e_j}(\vec x_j) = q_{\e_j}$,
and let
$r_j:=|\vec{ x}_j|$. It follows from the fact that
$q_j\in \Omega_{\e/3}$ that $r_j>\frac{1}{3}$.

There are several different ways in which $q_j$ can converge to $p_1$. We  consider now three subcases, and
use Lemma \ref{general} in each subcase.

{\sc Case 3a:} There are uniform upper and lower bounds on $r_j$ and ${y_j}/{r_j}$; i.e., for some $d>0$,
\begin{equation}\label{rnd17}
\frac{1}{3}<r_j\le d \text{\ \ and\ \ }1\ge  \frac{ y_j}{r_j}\ge\frac{1}{d}>0.
\end{equation}
In this case we take $(\Sigma,\gamma)=(\h^3 ,\ghyp)$,
$N_j:=\mathcal{A}_{\e_j}$ and $g_j:=\Psi_{\e_j}^*g_{\e_j}$.  It is immediate that $\bigcup_j N_j=\Sigma$, and that every compact set in $\Sigma$ is contained in almost all $N_j$.
It follows from \eqref{metric} that
$g_j\to \breve{g}$ uniformly on compact sets together with two derivatives.

Consider the function $\xi(y,x)=yF(r)$ and vector fields $Y_j$ such that $(\Psi_{\e_j})_* Y_j= \e_j^{-\delta}X_j$.   Because
$\Psi_{\e_j}^* \rho_{\e_j} = \e_j\xi$, we have
\begin{equation*}
|Y_j|_{g_j} \le \xi^\delta \qquad \text{and} \qquad
\sup_{N_j}\left(\xi^{-\delta}|\vlap_{g_j}Y_j|_{g_j}\right)\to 0,
\end{equation*}
so conditions
\eqref{Yj-bdd-above} and
\eqref{LYj-bdd-above}
of Lemma \ref{general} are satisfied with $\zeta \equiv\xi$.
The compact set $K_0$ characterized by (\ref{rnd17})
contains the points $( \vec x_j)$, where we have
$$|Y_j(\vec x_j)|_{g_j}=
\e_j^{-\delta}|X_j(q_{\e_j})|_{g_{\e_j}}\ge\tfrac12 { \e_j^{-\delta}\rho_{\e_j}(q_{\e_j})^\delta}
= \tfrac12 \xi(\vec x_{j})^\delta.$$
This means that the condition \eqref{Yj-bdd-below} of Lemma \ref{general} also holds.
Therefore,  there exists a nonzero $C^\infty$ vector field on $Y$ on $\h^3 $ that satisfies
$\lhyp Y=0$ and $|Y|_{\ghyp}=O(\xi^\delta)$.
However, this contradicts Proposition \ref{eugene'}.

{\sc Case 3b:} There is a uniform positive lower bound on ${y_j}/{r_j}$, but $r_j$ are unbounded. Passing to a subsequence, we may assume that
\begin{equation*}
r_j\to \infty, \qquad
\frac{ y_j}{r_j}\ge \frac{1}{d}>0.
\end{equation*}
We  again take $(\Sigma,\gamma)=(\h^3 ,\ghyp)$. Consider the transformation
$\mathcal{T}_j:\h^3 \to \h^3 $ given by $\mathcal{T}_j(\vec{ x})=r_j \vec{ x}.$
This transformation is chosen so that the points
$\vec a_j:=\mathcal{T}_j^{-1}( \vec x_j)=(y_j/r_j,x^1_j/r_j,x^2_j/r_j)$ lie in a compact region $K_0$ of the upper hemisphere $\{r=1,  y >0\}$.

The set $N_j\subseteq \h^3 $ characterized by
\begin{equation*}
N_j :=\left\{ \vec x: \frac{1}{3r_j}< r< \frac{1}{\e_j r_j}\right\}
\end{equation*}
is taken via $\mathcal{T}_j$ to the outer portion of the expanding annulus $\mathcal{A}_{\e_j}$.
Since in this case  $r_j\to\infty$ and $\e_j r_j\to 0$, we see that $\bigcup_j N_j=\h^3$ and that any compact subset of $\h^3$ is contained in almost all $N_j$.
Define $g_j:=\mathcal{T}_j^*\Psi_{\e_j}^*g_{\e_j}$; because $\e_j r_j\to 0$,
a simple argument using
\eqref{metric} shows
that $g_j$ converges uniformly to $\ghyp$ on compact subsets of $\h^3$ together with two derivatives.

Consider
$$\zeta= y=\frac{1}{\e_j r_jF(r_jr)}\mathcal{T}_j^*\Psi_{\e_j}^* \rho_{\e_j} \text{\ \ \ and\ \ \ }
Y_j:=(\e_jr_j)^{-\delta}\mathcal{T}_j^*\Psi_{\e_j}^*X_j.$$
Since $F$ is bounded above and below by positive constants on $[1/3,\infty)$, we have
\begin{align*}
y^{-\delta}|Y_j|_{g_j}&= (\e_j r_j y)^{-\delta}\Big(|X_j|_{g_{\e_j}}\circ \Psi_{\e_j}\circ \mathcal{T}_j\Big)\le
(\e_j r_j y)^{-\delta} \bigl( T_j^* \Psi_{\e_j}^*\rho_{\e_j}\bigr)^\delta =
F(r_jr)^{\delta} \le C;\\
y^{-\delta}|\vlap_{g_j}Y_j|_{g_j}&=
(\e_j r_j y)^{-\delta}\Big(|\vlap_{g_{\e_j}}X_j|_{g_{\e_j}}\circ \Psi_{\e_j}\circ \mathcal{T}_j\Big)
\le F(r_jr)^\delta\|\vlap_{\e_j}X_j\|_{0,0, \delta}\to 0.
\end{align*}
Moreover, since
$\Psi_{\e_j}\circ\mathcal T_j(\vec a_j) = q_{\e_j}$
and
$\rho_{\e_j}(q_{\e_j}) = \Psi_{\e_j}^*\rho_{\e_j}(\vec x_j) =\e_j y_j F(r_j)$,
we have
\begin{equation*}
|Y_j(\vec a_j)|_{g_j}= (\e_jr_j)^{-\delta}|X_j(q_{\e_j})|_{g_{\e_j}}\ge
\tfrac12 (\e_jr_j)^{-\delta} \rho_{\e_j}(q_{\e_j})^\delta
=
\tfrac12 (y_j/r_j)^\delta F(r_j)^\delta
\ge C y(\vec a_j)^\delta.
\end{equation*}
Consequently, the conditions of Lemma \ref{general} are fulfilled.
This means that we now have a $C^\infty$ nonzero vector field $Y$ on $\h^3$  for which
$\lhyp Y=0$ and $|Y|_{\ghyp}=O( y^\delta)$.  This is
a contradiction to Corollary \ref{eugene}.

{\sc Case 3c:} There is no positive lower bound on ${y_j}/{r_j}$. Passing to a subsequence, we may assume that
one of the following holds:
\begin{enumerate}\romanletters
\item\label{rj->infinity}
$r_j\to \infty$, or
\item\label{yj->0}
$y_j\to 0$ and $r_j$ is bounded.
\end{enumerate}
Note that in both cases \eqref{rj->infinity} and \eqref{yj->0}, $|(  x^1_{j},   x^2_{j})|/y_j\to \infty$.
In either case, we consider transformations
$\mathcal T_j\colon \h^3\to \h^3$ defined by
\begin{equation*}
\mathcal{T}_j( y, x^1, x^2)=( y_j y,  y_j x^1+ x^1_{j},  y_j x^2+x^2_{j}),
\end{equation*}
chosen so that $\Psi_{\e_j}\circ \mathcal{T}_j(1,0,0)=q_{\e_j}$.
Let $N_j:=\mathcal{T}_j^{-1}(\mathcal{A}_{\e_j}^+)$, where $\mathcal{A}_{\e_j}^+:=\{r>\frac{1}{6}\}\cap \mathcal{A}_{\e_j}$.
The region $N_j\subseteq \h^3$ is the semiannular region centered at $(0,-x^1_{j}/{ y_j}, -{x^2_{j}}/{y_j})$ of inner radius ${1}/{6y_j}$ and outer radius
$1/{\e_j y_j}$.

In case \eqref{rj->infinity},
we have $|(  x^1_{j},   x^2_{j})|\to \infty$ and
\begin{equation*}
\left|\left(\frac{- x^1_j}{y_j}, \frac{-x^2_j}{y_j}\right)\right| - \frac{1}{6y_j}
= \frac{|(  x^1_{j},   x^2_{j})|-\frac{1}{6}}{ y_j}\ge
\frac{|(  x^1_{j},   x^2_{j})|}{2 y_j}\to \infty.
\end{equation*}
In case \eqref{yj->0},
once $j$ is big enough, we have $|(   x^1_{j},    x^2_{j})|-\frac{1}{6}\ge\frac{1}{7}$, as a consequence of $r_j> \tfrac{1}{3}$. It follows that
\begin{equation*}
\left|\left( \frac{-x^1_j}{y_j}, \frac{-x^2_j}{y_j}\right)\right| - \frac{1}{6y_j}
=
\frac{|(   x^1_{j},    x^2_{j})|-\frac{1}{6}}{  y_j}\ge
\frac{1}{7  y_j}\to \infty.
\end{equation*}
In particular, $N_j$ contains the half-ball of radius
\begin{equation*}
R_j=\min\left\{\frac{1}{\e_jy_j} - \left|\left(\frac{- x^1_j}{y_j}, \frac{-x^2_j}{y_j}\right)\right|,\
\left|\left( \frac{-x^1_j}{y_j}, \frac{-x^2_j}{y_j}\right)\right| - \frac{1}{6y_j} \right\}
\end{equation*}
centered at the origin. Since $\e_j\vec{ x}_j\to 0$, the first expression in the minimum also converges to $\infty$. Therefore, $R_j\to +\infty$, $\bigcup_j N_j=\h^3$, and every compact subset of $\h^3$ is contained in almost all $N_j$.

As in the previous case, we  take $(\Sigma, \gamma)=(\h^3,\ghyp)$ and $g_j:=\mathcal{T}_j^*\Psi_{\e_j}^*g_{\e_j}$.
Note that \eqref{metric} shows that
$g_j$
can be expressed as
\begin{equation*}
g_j = \ghyp +
k_{ab,\e_j}\bigl( \mathcal T_j(y,x)\bigr)\frac{dx^a}{y}\,\frac{dx^b}{y},
\end{equation*}
where
$k_{ab,\e}$ are functions defined on $\A_\e$ such that $\|k_{ab,\e}\|_{C^{2,\alpha}(\A_c)}\to 0$
for any fixed $c>0$.  We need to show that $\|k_{ab,\e}\circ\mathcal T_j\|_{C^{2,\alpha}(K)}\to 0$
for any fixed compact set $K\subset \h^3$.
We will consider cases   \eqref{rj->infinity} and
\eqref{yj->0} separately.
Let $K\subset\h^3$ be a compact set, and let $R$ be the supremum of $r|_K$.

First assume we are  in case  \eqref{rj->infinity}.
For any point $\vec x = (y,x)\in K$, as soon as $j$ is large enough that $R<\tfrac12 |(x^1_j,x^2_j)|/y_j$,
the reverse triangle inequality gives
\begin{align*}
r\circ \mathcal T_j(\vec x)  = y_j \left| \left( 0, \frac{x^1_j}{y_j}, \frac{x^2_j}{y_j}\right)+\vec x  \right|
\ge
y_j \left(\frac{|(x^1_j,x^2_j)|}{y_j} - R\right)
\ge \tfrac12 y_j \frac{|(x^1_j,x^2_j)|}{y_j} = |(x^1_j,x^2_j)|\to \infty.
\end{align*}
Eventually, therefore, the set $T_j(K)$ lies in the portion of $\A_\e$ where $r>2$, and
thus for $\vec x\in K$ we have
\begin{align*}
k_{ab,\e_j}\bigl( \mathcal T_j(\vec x)\bigr) =
m_{ab,1} (\e_j y_j y, \e_j y_j x^1 + \e_j x^1_j, \e_j y_j x^2 + \e_j x^2_j),
\end{align*}
and it follows easily that $\|k_{ab,\e_j}\circ\mathcal T_j\|_{C^{2,\alpha}(K)}\to 0$
because $(\e_jy_j, \e_j x^1_j, \e_j x^2_j)\to 0$.

On the other hand, if we are in case \eqref{yj->0},
then  for any $\vec x\in K$, as soon as $R<|(x^1_j,x^2_j)|/y_j$, we have
\begin{align*}
r\circ \mathcal T_j(\vec x)
\le  y_j \left( \frac{|(x^1_j,x^2_j)|}{y_j}+R\right)
\le 2 |(x^1_j,x^2_j)| \le 2   r_j\le C,
\end{align*}
since $\{r_j\}$ is bounded.  Therefore, $T_j(K)$ is contained in the
fixed annulus $\{\vec x: \tfrac 1 6 < r < C\}$, on which
$k_{ab,\e}$ converges to zero in $C^{2,\alpha}$ norm.
Because $y_j\to 0$, the transformations $T_j$ are affine transformations
with uniformly bounded Jacobians, and so again we conclude that
${\|k_{ab,\e}\circ\mathcal T_j\|_{C^{2,\alpha}(K)}}\to 0$.
In both cases, therefore, $g_j\to \ghyp$ in $C^{2,\alpha}(K)$.

This time, we let
\begin{equation*}
\zeta= y=\frac{1}{\e_jy_jF(r\circ T_j)}\mathcal{T}_j^*\Psi_{\e_j}^* \rho_{\e_j} \text{\ \ \ and\ \ \ }
Y_j:=(\e_jy_j)^{-\delta}\mathcal{T}_j^*\Psi_{\e_j}^*X_j.
\end{equation*}
Reasoning as in Case 3b, since $F$ is bounded above on $[\tfrac16,\infty)$ and bounded
below everywhere, we find that
\begin{align*}
y^{-\delta}|Y_j|_{g_j}&= (\e_j y_j y)^{-\delta}\Big(|X_j|_{g_{\e_j}}\circ \Psi_{\e_j}\circ \mathcal{T}_j\Big)\le
(\e_j y_j y)^{-\delta} \bigl( T_j^* \Psi_{\e_j}^*\rho_{\e_j}\bigr)^\delta =
F(r\circ T_j)^{\delta} \le C;\\
y^{-\delta}|\vlap_{g_j}Y_j|_{g_j}&=
(\e_j y_j y)^{-\delta}\Big(|\vlap_{g_{\e_j}}X_j|_{g_{\e_j}}\circ \Psi_{\e_j}\circ \mathcal{T}_j\Big)
\le F(r\circ T_j)^\delta\|\vlap_{\e_j}X_j\|_{0,0, \delta}\to 0;\\
|Y_j(1,0,0)|_{g_j}&= (\e_j y_j)^{-\delta}|X_j(q_{\e_j})|_{g_{\e_j}}\ge
\tfrac12 (\e_j y_j)^{-\delta} \rho_{\e_j}(q_{\e_j})^\delta
=
\tfrac12 y_j^{-\delta} \bigl(y_j F(y_j)\bigr)^\delta
\ge  c.
\end{align*}
These estimates show that the vector fields $Y_j$ satisfy the conditions of Lemma \ref{general} for the choice of $K_0=\{(1,0,0)\}$. We now see that there exists a nonzero $C^\infty$ vector field $Y$ on $\h^3$ such that
$$\lhyp Y=0, \ \ \ Y=O( y^\delta) \text{\ \ \ for \ \ \ }\delta\in (-1, 3).$$
However, this is impossible by Corollary \ref{eugene}.
\end{proof}

Now we are ready for our main theorem concerning the vector Laplacian.

\begin{theorem}\label{vectorlaplacian}
If $\e$ is sufficiently small and if $-1<\delta<3$, $0<\alpha<1$, then  the vector Laplacian
$$\vlap_\e:C^{2,\alpha}_\delta(M_\e)\to C^{0,\alpha}_\delta(M_\e)$$
is invertible and the norm of its inverse is bounded uniformly in $\e$.
\end{theorem}

\begin{proof}
The invertibility follows from the analysis in
\cite {IMP1} or  \cite{lee}, and the uniform estimate
follows by combining
\eqref{eq:schauder-vlap} with
\eqref{eq:unif-C0}.
\end{proof}

\section{Correcting the traceless part of the second fundamental form}
\label{sec:correcting-II}

In this section, we  use the elliptic PDE theory and analysis of the previous section
to add a correction to our spliced  tensor $\mu_\e$ to make it divergence-free.  First we show
that its divergence is not too large.

\begin{lemma}\label{lemma:div-not-too-large}
With $\mu_\e$ defined by \eqref{eq:def-mue},
$\divergence_{g_\e}\mu_\e\in C^{0,\alpha}_1(M_\e)$ with norm
$\|\divergence_{g_\e}\mu_\e\|_{0,\alpha,1}=O\bigl(\sqrt{\e}\bigr)$.
\end{lemma}

\begin{proof}
Recall that we have defined $\widehat\mu_\e = \chi_\e\mu$,
where $\mu$ is the given traceless second fundamental form and $\chi_\e$
is defined by \eqref{eq:def-chi}.
Restricted to the support of $\widehat \mu_\e$, the projection
$\pi_\e$ is a
diffeomorphism taking $g$ to $g_\e$ and $\widehat \mu_\e$ to $\mu_\e$,
so it suffices to show that $\divergence_{g}\widehat\mu_\e\in C^{0,\alpha}_1(M)$ with
$O(\sqrt{\e})$ norm.

For a vector field $Y$ and a symmetric $2$-tensor $\eta$, let us use
the notation $Y \into\eta$ to denote the $1$-form $\eta(Y,\cdot)$.
It is easy
to check (by doing the computation in M\"obius coordinates) that
the map $(Y,\eta)\mapsto Y\into \eta$ is a continuous bilinear map
from $C^{k,\alpha}_{\delta_1}(M)\times C^{k,\alpha}_{\delta_2}(M)$
to $C^{k,\alpha}_{\delta_1+\delta_2}(M)$ for any $\delta_1,\delta_2\in\R$.

It follows easily from the definition of the divergence operator and the fact
that $\mu$ is divergence-free that
\begin{equation*}
\divergence _g \widehat\mu_\e = \chi_\e \divergence_g \mu + (\grad_g \chi_\e)\into\mu = (\grad_g\chi_\e)\into\mu.
\end{equation*}
The support of $\grad_g(\chi_\e)$ is contained in the union of the two half-balls $\overline B_{1,1}\cup\overline B_{1,2}$.
Letting $\theta^j$ denote either $\theta^j_1$ or $\theta^j_2$ depending on which half-ball
we are in, we compute
\begin{equation*}
\grad_g\chi_\e
= \chi'\left( \frac{\rho^2}{\e^2} + \sum_j \frac{(\theta^j)^2}{\e}\right)
\left( \frac{2\rho \grad_g\rho}{\e^2} + \sum_j \frac{2\theta^j \grad_g\theta^j}{\e}\right),
\end{equation*}
and therefore
\begin{equation}\label{eq:divg-mue}
\divergence _g \widehat\mu_\e =
\chi'\left( \frac{\rho^2}{\e^2} + \sum_j \frac{(\theta^j)^2}{\e}\right)
\left( \frac{2\rho }{\e^2} (\grad_g \rho)\into\mu + \sum_j\frac{ 2\theta^j}{\e} (\grad_g\theta^j)\into\mu\right).
\end{equation}

Using the formula for the change in Christoffel
symbols under a conformal change in metric (see, for example, \cite[equation (3.10)]{lee}),
we find that
\begin{equation*}
0 = \divergence_g\mu
=  \rho^2\divergence_{\overline g}\mu- \rho(\grad_{\overline g}\rho)\into\mu.
\end{equation*}
After substituting $\mu= \rho^{-1}\overline\mu$, this becomes
\begin{equation*}
0
= \rho^2\divergence_{\overline g}(\rho^{-1}\overline\mu) - \rho(\grad_{\overline g}\rho)\into(\rho^{-1}\overline\mu)
= \rho\divergence_{\overline g}\overline\mu - 2 (\grad_{\overline g}\rho) \into{\overline \mu}.
\end{equation*}
It follows that
$(\grad_{\overline g}\rho) \into{\overline \mu} = \tfrac12 \rho\divergence_{\overline g}\overline\mu$,
and thus
\begin{align*}
(\grad_{g}\rho)\into\mu
&= (\rho^2 \grad_{\overline g}\rho)\into(\rho^{-1}\overline \mu)\\
&=\tfrac12 \rho^2\divergence_{\overline g}\overline\mu.
\end{align*}
Since $\divergence_{\overline g}\overline \mu$ is a $1$-form
whose coefficients in background coordinates are in $C^{0,\alpha}(\overline M)$,
$\divergence_{\overline g}\overline \mu$ is contained in $C^{0,\alpha}_1(M)$, and thus $(\grad_{g}\rho)\into\mu\in C^{0,\alpha}_{3}(M)$.
On the other hand, a straightforward computation shows that $\grad_g\theta^j\in C^{0,\alpha}_1(M)$,
and therefore $(\grad_g\theta^j)\into \mu\in C^{0,\alpha}_2(M)$. We conclude  that
the following quantities are finite:
\begin{equation*}
\| \rho^{-2} (\grad_g\rho)\into\mu\|_{0,\alpha,1}, \qquad
\| \rho^{-1} (\grad_g\theta^j)\into\mu\|_{0,\alpha,1} \quad (j=1,2).
\end{equation*}

With this in mind, we rewrite \eqref{eq:divg-mue} as
\begin{equation}\label{eq:divg-mue-2}
\divergence _g \widehat\mu_\e =
f_\e(\rho,\theta)
\left( \frac{2\rho^3 }{\e^2} \rho^{-2}(\grad_g \rho)\into\mu + \sum_j\frac{ 2\theta^j\rho}{\e} \rho^{-1}(\grad_g\theta^j)\into\mu\right),
\end{equation}
where
\begin{equation*}
f_\e(\rho,\theta) = \chi'\left( \frac{\rho^2}{\e^2} + \sum_j \frac{(\theta^j)^2}{\e}\right).
\end{equation*}
Note that
$f_\e$ is bounded independently of $\e$, and is supported
in a region where
$\rho\le \e\sqrt{3}$ and $|\theta^j|\le \sqrt{3\e}$.
Its differential satisfies
\begin{equation*}
df_\e = \chi''\left( \frac{\rho^2}{\e^2} + \sum_j \frac{(\theta^j)^2}{\e}\right)
\left( \frac{2\rho d\rho}{\e^2} + \sum_j \frac{2\theta^j d\theta^j}{\e}\right).
\end{equation*}
Because
 $|d\rho|_g$ and $|d\theta^j|_g$ are both bounded by multiples
of $\rho$, it follows that $|df_\e|_g$ is bounded uniformly in $\e$.
Therefore, $f_\e$ is uniformly bounded in $C^1(M)$ and
thus also in $C^{0,\alpha}_0(M)$.
Inserting these
estimates into
\eqref{eq:divg-mue-2}, we find that
\begin{align*}
\|\divergence _g &\widehat\mu_\e\|_{0,\alpha,1} \\
&\le
\|f_\e\|_{0,\alpha,0}\biggl(\frac{2(\e\sqrt{3})^3}{\e^2}
\left\| \rho^{-2}(\grad_g\rho)\into\mu\right\|_{0,\alpha,1} +
\sum_j
\frac{2(\sqrt{3\e})(\e\sqrt{3})}{\e}
\left\| \rho^{-1}(\grad_g\theta^j)\into\mu\right\|_{0,\alpha,1}
\biggr)\\
&\le C'(\e + \sqrt{\e})\\
&\le C''\sqrt{\e}.
\end{align*}
This completes the proof.
\end{proof}

Using the preceding result and Theorem  \ref{vectorlaplacian}, 
the idea now is to make a small perturbation of $\mu_\e$, which we denote by $\nu_\e$, for which $\mathrm{div}_{g_\e} \nu_\e=0$.

\begin{theorem}\label{thm:nu-convergence}
For each sufficiently small $\e>0$, 
there is a polyhomogenous symmetric $2$-tensor field $\nu_\e$, which
is traceless and divergence-free with respect to $g_\e$, 
such that $\rho^2\nu_\e$ has a $C^2$ extension to $\Mbar_\e$
that vanishes on $\partial\Mbar_\e$, and $\nu_\e$ 
satisfies
\begin{equation}\label{main-mu}
\| \nu_\e - \mu_\e\|_{1,\alpha, 1}=O(\sqrt{\e}).
\end{equation}
In particular, away from the neck,  $\nu_\e$ converges uniformly in $C^{1,\alpha}_1$ 
\(and therefore also in $C^{1,\alpha}$\)
to
\(the projection of\) $\mu$.
\end{theorem}

\begin{proof}
We use the standard technique for finding a divergence-free perturbation of $\mu_\e$ discussed in Section \ref{sec:prelim} (see \eqref{babyL}). Relying on Theorem \ref{vectorlaplacian},
for each small $\e>0$, we let $X_\e$ be the unique vector field in $C^{2,\alpha}_1(M_\e)$
that satisfies
\begin{equation}\label{eq:vlap-eqn}
\vlap_{\e}X_\e=(\divergence_{g_\e} \mu_\e)^{\sharp},
\end{equation}
and we set 
$\nu_\e:=\mu_\e+\mathcal{D}_\e X_\e$.  By definition of the conformal Killing operator,
$\nu_\e$ is traceless; and by construction
it is divergence-free.
Since $\|\mathrm{div}_{g_\e}\mu_\e\|_{0,\alpha,1}=O(\sqrt{\e})$,
it follows from the uniform estimate of Theorem \ref{vectorlaplacian} that
$\|X_\e\|_{2,\alpha,1}=O(\sqrt{\e})$. The arguments of Section \ref{holder} show that
$\mathcal{D}_\e$ is bounded from
$C^{2,\alpha}_\delta(M_\e)$ to $C^{1,\alpha}_\delta(M_\e)$ uniformly in $\e$, 
and therefore 
\eqref{main-mu}  is satisfied. It now remains to show that 
$\nu_\e$ is polyhomogeneous and that $\rho_\e^2\nu_\e$ has a $C^2$ extension to $\Mbar_\e$ which vanishes on the ideal boundary. 

We start by observing that the right hand side of \eqref{eq:vlap-eqn} is polyhomogeneous and that, by Theorem 6.3.10 of \cite{piot}, there exists a polyhomogeneous solution $X^{\mathrm{phg}}_\e$ of \eqref{eq:vlap-eqn}. Since $\mu$ has an asymptotic expansion beginning with $\rho^{-1}$ and the first log term (if any) appearing with $\rho^s$, $s>0$, a computation shows that the vector field on the right-hand side of  \eqref{eq:vlap-eqn} has an expansion beginning with a $\rho_\e^2$ term, and with the first log terms (if any) appearing in the $\rho_\e^{s+3}$ term, $s>0$.  
Inserting the general asymptotic expansion for $X^{\mathrm{phg}}_\e$ into 
\eqref{eq:vlap-eqn} and matching like terms inductively, we conclude that 
$X^{\mathrm{phg}}_\e$ has an asymptotic expansion beginning with $\rho_\e^2$ 
and the first log terms appearing with $\rho_\e^{s+3}$, $s>0$. (Note that the first log terms which arise from the indicial roots of $L_\e$ appear with $\rho_\e^4$.) On the other hand, 
$X_\e\in C^{2,\alpha}_1(M_\e)$ implies that its component functions in background
coordinates are also $O(\rho_\e^2)$. The  
uniqueness part of Theorem 6.3.10 
in \cite{piot} implies that 
$X_\e=X^{\mathrm{phg}}_\e$. It now follows easily 
that $\nu_\e$ is polyhomogeneous and that is has an asymptotic expansion starting with $\rho_\e^{-1}$ and the first log term appearing with $\rho_\e^s$, $s>0$. Thus the extension of $\rho_\e^2\nu_\e$ 
to $\Mbar_\e$ is actually of class $C^2$ and vanishes on the ideal boundary,
which is just what is needed for $\nu_\e$ to be the traceless part
of the second fundamental form for a polyhomogeneous AH initial data set
(see Definition \ref{def:phg-dataset}).
\end{proof}

\section{The Lichnerowicz Equation and Conformal Deformation to the Spliced Solutions of the Constraint Equations}\label{lish}

Thus far, starting with a set of asymptotically hyperbolic, polyhomogeoneous, constant mean curvature initial data $(M,g,K)$ satisfying the Einstein constraint equations, together with a pair of points $\{p_1, p_2\}$ both contained in the ideal boundary $\partial \Mbar$, we have first produced a one-parameter family of spliced data sets $(M_\e, g_\e, \mu_\e)$ which are asymptotically hyperbolic, polyhomogeneous,
CMC, and not solutions of the constraints, and we have then corrected $\mu_\e$ to a new family of symmetric tensors $\nu_\e$ which are all divergence free as well as trace free with respect to $g_\e$.  We have verified that outside of the gluing region, $\nu_\e$ approaches the original trace free part of $K$ in an appropriate sense. 

To complete our gluing construction, we will now carry out a one-parameter family  of conformal deformations  that transform the data $(M_\e, g_\e, \nu_\e, \tau=3)$ to a family of data sets $(M_\e, \psi_\e^4 g_\e, \psi_\e^{-2} \nu_\e + \psi_\e^4 g_\e )$
satisfying the desired properties of the gluing construction (including the constraint equations) for all $\e$. Following the principles of the conformal method outlined in Section \ref{sec:conf-method}, if we want the conformally transformed  data sets to satisfy the constraints, then the conformal functions
$\psi_\e$ must solve the Lichnerowicz equation (\ref{LichneroeqW=0}), which for the data $(M_\e, g_\e, \nu_\e, \tau=3)$ takes the form $\Lie_\e(\psi_\e)=0$, where
\begin{equation}\label{Leqn}
\Lie_\e(u):=\Delta_{g_\e}u-\frac{1}{8}R(g_\e) u+\frac{1}{8}|\nu_\e|_{g_\e}^2u^{-7}-\frac{3}{4}u^5.
\end{equation}
Hence, we need to do the following: prove that for each $\e$ the Lichnerowicz equation (\ref{Leqn}) does admit a 
 positive solution $\psi_\e$ which is polyhomogeneous 
and $C^2$ up to the ideal boundary,
  prove that $\psi_\e$ approaches $1$ at the ideal boundary (so that the resulting Riemannian manifold is AH) and prove that as $\e \rightarrow 0$ the solutions $\psi_\e$ approach $1$ away from the gluing region.  We carry out these proofs here.

\bigbreak

The first step in our proof that, for the data sets $(M_\e, g_\e, \nu_\e, \tau=3)$, the Lichnerowicz equation admits solutions with the desired asymptotic properties, is to estimate the extent to which the constant function $\psi_0\equiv 1$ fails to be a solution of \eqref{Leqn}. While it is relatively straightforward to show that $\Lie_\e(\psi_0)$ is an element of the weighted H{\"o}lder space $C^{0,\alpha}_1(M_\e)$, we have been unsuccessful in proving that the corresponding norm of $\Lie_\e(\psi_0)$ is ``small".  Consequently, we are able to find a solution $\psi_\e$ of the Lichnerowicz equation such that $\psi_\e-\psi_0$ ``vanishes" on $\partial\Mbar_\e$ but we are only able to obtain good estimates on $\psi_\e-\psi_0$ in $C^{2,\alpha}(M_\e)$. This, however, is sufficient to  prove our main result. 

\begin{lemma} \label{N(1)estimate}
We have $\Lie_\e(\psi_0)\in C^{0,\alpha}_1(M_\e)$ and 
$\|\Lie_\e(\psi_0)\|_{0,\alpha}=O(\sqrt{\e})$ as $\e\to 0$.
\end{lemma}

\begin{proof}
Note that $\psi_0\equiv 1$ yields
\begin{equation*}
\Lie_\e(\psi_0)=-\frac{1}{8}\left(R(g_\e) + 6-|\nu_\e|_{g_\e}^2\right).
\end{equation*}
The fact that $\Lie_\e(\psi_0)\in C^{0,\alpha}_1(M_\e)$ is immediate from  $|\nu_\e|^2_{g_\e}\in C^{0,\alpha}_1(M_\e)$ which is true by construction, and $(R(g_\e)+6)\in C^{0,\alpha}_1(M_\e)$ which is true by virtue of   \eqref{scalarcurv:comp} and Lemma \ref{scalcurv:est1}. 
To estimate the unweighted norm of $\Lie_\e(\psi_0)$,
let $\Phi$ be one of our preferred charts for $M_\e$ (see Section \ref{holder}). If the image of $\Phi$ is away from the gluing region, i.e., if $\Phi({\breve B})\cap \Psi_\e\bigl(\A_{\sqrt{\e}/{3}}\bigr)= \emptyset$, then
$$\|\Phi^* \Lie_\e(\psi_0)\|_{C^{0,\alpha}({\breve B})}=\tfrac{1}{8}\|\Phi^*(|\nu_\e|_{g_\e}^2-|\mu_\e|_{g_\e}^2)\|_{C^{0,\alpha}({\breve B})}=O(\sqrt{\e})$$
as a consequence of the second constraint equation $R(g)-|K|_g^2+\tau^2=R(g)-|\mu|_g^2+6=0$ and (\ref{main-mu}). Thus, it remains to study the charts $\Phi$  for which $$\Phi({\breve B})\subseteq \Psi_\e\big(\A_{c\sqrt{\e}}\big),$$
where $c>0$ is some sufficiently small fixed number. We start with the inequality
$$\|\Phi^* \Lie_\e(\psi_0)\|_{C^{0,\alpha}({\breve B})}\le \tfrac{1}{8}\|\Phi^* (R(g_\e)+6)\|_{C^{0,\alpha}({\breve B})}+O(\sqrt{\e})+\tfrac{1}{8}\|\Phi^* |\mu_\e|_{g_\e}^2\|_{C^{0,\alpha}({\breve B})}.$$
It follows from  \eqref{scalarcurv:comp} and Lemma \ref{scalcurv:est1} that 
$\|R(g_\e)+6\|_{0,\alpha,1}$
is bounded uniformly in $\e$.
The uniformity properties \eqref{unif-rho-est} and the fact that $\rho_\e\le \sqrt{\e}/c$ on $\Psi_\e\left(\mathcal{A}_{c\sqrt{\e}}\right)$  imply that 
\begin{equation}\label{scalcurv:est2}
\|\Phi^* (R(g_\e)+6)\|_{C^{0,\alpha}({\breve B})}=O(\sqrt{\e}).
\end{equation}
To understand  the $\mu_\e$-term, note that $\pi_\e$ is a diffeomorphism on the support of $\widehat \mu_\e$, where we also have
$$\big|\widehat \mu_\e\big|_g^2\le \big|\mu\big|_g^2=\rho^2\big|\overline{\mu}\big|_{\bar{g}}^2.$$
It then follows that $\displaystyle{\sup_{\breve B}}\left|\Phi^* |\mu_\e|_{g_\e}^2\right|=O(\e)$. Likewise,
$$
\big|d(|\widehat \mu_\e|_g^2)\big|_g=\rho\big|d(\rho^2\chi^2|\overline \mu|_{\bar{g}}^2)\big|_{\bar g}
\le \rho\Big(2\rho|d\rho|_{\bar g}|\overline \mu|_{\bar{g}}^2+2\rho^2|d\chi|_{\bar g}|\overline \mu|_{\bar{g}}^2+\rho^2|d(|\overline \mu|_{\bar{g}}^2)|_{\bar g}\Big)=O(\rho^2)$$
as a consequence of the  boundedness of $|d\rho|_{\bar g}$ and the fact that $\rho^2|d\chi|_{\bar g}=\rho O(\e)O(\tfrac{1}{\e})=O(\rho)$ on the support of $d\chi$. Overall, we see that
$$
\|\Phi^* |\mu_\e|_{g_\e}^2\|_{C^{0,\alpha}({\breve B})}=O(\e)
$$
and therefore $\|\Phi^* \Lie_\e(\psi_0)\|_{C^{0,\alpha}({\breve B})}=O(\sqrt{\e})$ as $\e\to 0$.
\end{proof}

It should also be pointed out that $|\nu_\e|^2_{g_\e}=R(g_\e)+6+8\Lie_\e(\psi_0)$ implies that \begin{equation}\label{nu:est}
\left\||\nu_\e|^2_{g_\e}\right\|_{0,\alpha}\le C
\end{equation}
for some $C>0$.

\bigbreak

The main ingredient in our study of the solvability and the solutions of the Lichnerowicz equation is  the uniform invertibility of the linearizations
$$\linP_\e:=\Delta_{g_\e}-\frac{1}{8}\Big(R(g_\e)+7|\nu_\e|^2_{g_\e}+30\Big)$$
of $\Lie_\e$ at $\psi_0=1$. 
In what follows we rely heavily on maximum principle(s). Part of the reason why this approach is successful is that the function 
$$f_\e:=\frac{1}{8}\Big(R(g_\e)+7|\nu_\e|^2_{g_\e}+30\Big)$$
has a positive lower bound. 

\begin{lemma}\label{fest:lemma}
Let  $C<3$ be a positive constant and let $\e$ be sufficiently small. We have
$$f_\e\ge C, \text{\ \ pointwise}.$$ 
\end{lemma}

\begin{proof}
It is enough to show 
\begin{equation}\label{f:est}
\sup_{M_\e}\left|f_\e-|\nu_\e|^2_{g_\e}-3\right|\to 0 \text{\ \ as\ \ }\e\to 0
\end{equation}
or, equivalently, that the sup-norms of $\left|R(g_\e)-|\nu_\e|^2_{g_\e}+6\right|$
both over the gluing region $\Psi_\e\bigl(\A_{\sqrt{\e}/3}\bigr)$ and over its complement $M_\e\smallsetminus \Psi_\e\bigl(\A_{\sqrt{\e}/3}\bigr)$
converge to $0$.
To prove this convergence on $\Psi_\e\bigl(\A_{\sqrt{\e}/3}\bigr)$, note that $\pi_\e$ maps diffeomorphically onto the support of $\mu_\e$, and that $$\left|\mu_\e\right|^2_{g_\e}\le \left(\rho^2 \left|\overline \mu\right|^2_{\bar g}\right)\circ\pi_\e^{-1}.$$ 
Thus, there is a constant $c>0$ such that 
$\left|\mu_\e\right|^2_{g_\e}\le c\e$ on $\Psi_\e\bigl(\A_{\sqrt{\e}/3}\bigr)$. In light of \eqref{main-mu} this means that 
$$\sup_{\Psi_\e\left(\A_{\sqrt{\e}/3}\right)}\left|\nu_\e\right|^2_{g_\e}\to 0 \text{\ \ as\ \ }\e\to 0.$$ The convergence result \eqref{f:est} on the gluing region now follows from Lemma \ref{scalcurv:est1} or, rather, estimate \eqref{scalcurv:est2}.
The convergence away from the gluing region
is an easy consequence of the fact that the restriction of $f_\e$ to $M_\e\smallsetminus \Psi_\e\bigl(\A_{\sqrt{\e}/3}\bigr)$ satisfies
$$\left(f_\e-|\nu_\e|^2_{g_\e}\right)\circ \pi_\e=\tfrac{1}{8}\left(R(g)+7\left|(\pi_\e)^*\nu_\e\right|^2_g+30\right)-\left|(\pi_\e)^*\nu_\e\right|^2_g=3+\tfrac{1}{8}\left(|\mu|^2_g-\left|(\pi_\e)^*\nu_\e\right|^2_g\right)=3+O(\sqrt{\e})$$ by virtue of \eqref{main-mu} and the second constraint equation 
$R(g)-|K|^2_g+9=R(g)-|\mu|^2_g+6=0$.
\end{proof}

Strictly speaking, the operators $\linP_\e$ are not ``geometric'' due to the presence of the $|\nu_\e|^2_{g_\e}$ term, so
the analysis of \cite{lee} does not apply directly.  There are many ways to 
circumvent this; for convenience, we will base our argument on 
Proposition 3.7 of \cite{GL}.  
First we need the following uniform estimate (called the ``basic estimate'' in \cite{GL}).
This estimate is analogous to 
the estimate of  Lemma \ref{vectLapl:lemma} above
for the vector Laplacian.  
The proofs of the two lemmas, however, are quite different: Lemma \ref{vectLapl:lemma} is proved using blow-up analysis, while the proof of the next lemma is direct and constructive. Consequently, the next lemma features a more optimal result on $C$ (as opposed to the vector Laplacian case where we are only able to prove the existence of $C$).

\begin{lemma}\label{LinLichnLemma}
Let $C>\tfrac{1}{3}$ be fixed, and assume $\e>0$ is sufficiently small.
\begin{enumerate}
\item If $u$ is a $C^2$ function on $M_\e$ with both $\rho_\e^{-1}u$ and $\rho_\e^{-1}\linP_\e u$ bounded then 
\begin{equation}\label{eq:linP-global-sup-est1}
\sup_{M_\e}|\rho_\e^{-1}u|\le C\sup_{M_\e}|\rho_\e^{-1}\linP_\e u|.
\end{equation}
\item If $u$ is a bounded $C^2$ function on $M_\e$ with $\linP_\e u$ bounded,
then
\begin{equation}\label{eq:linP-global-sup-est}
\sup_{M_\e}|u|\le C\sup_{M_\e}|\linP_\e u|.
\end{equation}
\end{enumerate}
Similarly, if $\Omega$ is a precompact subset of $M_\e$, and $u$ is a continuous function on $\overline\Omega$ that is $C^2$ in $\Omega$ and vanishes on
$\partial\Omega$,
then 
\begin{equation}\label{eq:linP-local-sup-est}
\sup_{\Omega}|\rho_\e^{-1}u|\le C\sup_{\Omega}|\rho_\e^{-1}\linP_\e u|, \text{\ \ and\ \ } \sup_{\Omega}|u|\le C\sup_{\Omega}|\linP_\e u|.
\end{equation}
\end{lemma}

\begin{proof} We start by proving \eqref{eq:linP-global-sup-est1}. Note that it suffices to consider functions $u$ for which 
$$\sup_{M_\e}|\rho_\e^{-1} u| = \sup_{M_\e}\left(\rho_\e^{-1} u\right).$$
Given a fixed $\e>0$ and a $C^2$-function $u\in C^{0,0}_1(M_\e)$,  Yau's Generalized Maximum Principle \cite{GL} implies that  there is a sequence of points $\{x_k\}$ of $M_\e$ such that 
\begin{enumerate}\romanletters
\item $\displaystyle{\lim_{k\to \infty}} \big[\rho_\e^{-1} u\big](x_k)=\sup_{M_\e} \big[\rho_\e^{-1} u\big]$
\item $\displaystyle{\lim_{k\to \infty}} \big|d(\rho_\e^{-1} u)\big|_{g_\e}(x_k)=0$
\item $\displaystyle{\limsup_{k\to \infty}}\  \Delta_{g_\e} \big[\rho_\e^{-1}u\big](x_k)\le 0$.
\end{enumerate}
Note that the condition (ii) can be re-written as
\begin{equation}\label{lim-du}
\displaystyle{\lim_{k\to \infty}} \Big|\rho_\e^{-1}du-\rho_\e^{-1}u\frac{d\rho_\e}{\rho_\e}\Big|_{g_\e}(x_k) =0.
\end{equation}
A short computation shows that 
\begin{equation}\label{laplace:comp}
\Delta_{g_\e}\big[\rho_\e^{-1}u\big]=\rho_\e^{-1}\Delta_{g_\e}u-2\<\frac{d\rho_\e}{\rho_\e}, \rho_\e^{-1}du-\rho_\e^{-1}u\frac{d\rho_\e}{\rho_\e}\>-\rho_\e^{-1}u\frac{\Delta_{g_\e} \rho_\e}{\rho_\e}.
\end{equation}
Recall that, by Lemma \ref{defnfnc}, the quantity $\Big|\frac{d\rho_\e}{\rho_\e}\Big|_{g_\e}$ is bounded for each fixed $\e>0$. Therefore, 
the identity (\ref{lim-du}) implies
\begin{equation}\label{limsup}
\displaystyle{\limsup_{k\to\infty}}\Big[\rho_\e^{-1}\Delta_{g_\e}u-\rho_\e^{-1}u\frac{\Delta_{g_\e} \rho_\e}{\rho_\e}\Big](x_k)=\displaystyle{\limsup_{k\to\infty}}\  \Delta_{g_\e}\big[\rho_\e^{-1}u\big](x_k)\le 0. 
\end{equation}
Since the defining functions $\rho_\e$ for small $\e>0$ are superharmonic (Lemma \ref{LaplaceRho}) and since $f_\e\ge \tfrac{1}{C}$ by Lemma \ref{fest:lemma}, we see that   
$$\rho_\e^{-1}\Delta_{g_\e}u-\rho_\e^{-1}u\frac{\Delta_{g_\e} \rho_\e}{\rho_\e}= \rho_\e^{-1}\linP_\e u+\rho_\e^{-1}\Big[f_\e-\frac{\Delta_{g_\e} \rho_\e}{\rho_\e}\Big]\ge \rho_\e^{-1}\linP_\e u +\tfrac{1}{C}\rho_\e^{-1}u.$$
Conditions (i) and \eqref{limsup} now imply
$$\displaystyle{\limsup_{k\to+\infty}}\big[\rho_\e^{-1}\linP_\e u\big](x_k) +\tfrac{1}{C}\sup_{M_\e} \big[\rho_\e^{-1}u\big]\le 0$$
for small enough $\e$.
Consequently, we have 
$$\sup_{M_\e} \big[\rho_\e^{-1}u\big]\le C\ \displaystyle{\liminf_{k\to +\infty}} \big[-\rho_\e^{-1}\linP_\e u\big](x_k)\le C \sup_{M_\e}|\rho_\e^{-1} \linP_\e u|,$$
as claimed. The proofs of the remaining three estimates are similar but considerably easier. Indeed, to prove \eqref{eq:linP-global-sup-est} we use $$\Delta_{g_\e}u(x_k)=\linP_\e u(x_k)+f_\e(x_k) u(x_k)\ge \linP_\e u(x_k)+\tfrac{1}{C}u(x_k)$$
in place of \eqref{laplace:comp},
while \eqref{eq:linP-local-sup-est} is proved using the ordinary maximum principle.
\end{proof}

\begin{theorem}\label{linLichn}
The operators $\linP_\e:C^{2,\alpha}_1(M_\e)\to C^{0,\alpha}_1(M_\e)$ and $\linP_\e:C^{2,\alpha}(M_\e)\to C^{0,\alpha}(M_\e)$ are invertible for sufficiently small $\e>0$. The norm of the  inverse of 
$\linP_\e:C^{2,\alpha}(M_\e)\to C^{0,\alpha}(M_\e)$
is bounded uniformly in $\e$.
\end{theorem}

\begin{proof}
It follows from Proposition 3.7 in \cite{GL} (together with Lemma \ref{LinLichnLemma})
that $\linP_\e$ is invertible when $\e$ is small enough,
so it remains only to prove uniformity of the norm.
We start by establishing a uniform elliptic estimate
\begin{equation}\label{Lichn:ellest}
\|u\|_{2,\alpha}\le C\left(\|\linP_\e u\|_{0,\alpha}+\|u\|_{0,0}\right)
\end{equation}
in which $C$ is independent of (sufficiently small) $\e>0$. 
Let $\Phi$ be one of our preferred charts for $M_\e$ (see section \ref{holder}). Consider the elliptic operator $\linP_{\Phi,\e}:C^{2,\alpha}(\breve B)\to C^{0,\alpha}(\breve B)$ defined by
$$\linP_{\Phi,\e}:=\Delta_{\Phi^* g_\e} - \Phi^*\Big(3+|\nu_\e|^2_{g_\e}-\Lie_\e(\psi_0)\Big);$$
this operator is of interest since
$$\Phi^*\left(\linP_\e u\right)=\linP_{\Phi,\e}\Phi^*u.$$
Recall that the metric $\Phi^* g_\e$ is uniformly equivalent to the hyperbolic metric $\breve g$. Furthermore, we see from Lemma \ref{N(1)estimate} and \eqref{nu:est} that the $C^{0,\alpha}(\breve B)$-norms of $\Phi^*|\nu_\e|^2_{g_\e}$ and $\Phi^* \Lie_\e(\psi_0)$ are uniformly bounded. Thus, the eigenvalues of the principal symbol of $\linP_{\Phi,\e}$ are uniformly bounded from below, while the $C^{0,\alpha}(\breve B)$-norms of the coefficients of $\linP_{\Phi,\e}$ are uniformly bounded from above. Let $\breve B_0$ be a fixed precompact subset of $\breve B$  such that  the restrictions of our preferred charts to $\breve B_0$ still cover $M_\e$.
It follows from the standard elliptic theory \cite{GT} that there is a constant $C$ (independent of $\e$, $\Phi$ and $u$) such that
$$\|\Phi^* u\|_{C^{2,\alpha}(\breve B_0)}\le C\left(\|\linP_{\Phi,\e} \Phi^*u\|_{C^{2,\alpha}(\breve B)}+\|\Phi^*u\|_{C^{0,0}(\breve B)}\right).$$
Taking the supremum with respect to $\Phi$ now yields \eqref{Lichn:ellest}.

Next, we combine Lemma \ref{LinLichnLemma} and the elliptic estimate \eqref{Lichn:ellest}.
We conclude that there is a constant $C$ (independent of $\e$) such that
$$\|u\|_{2,\alpha}\le C \|\linP_\e u\|_{0,\alpha}$$
for all  $u\in C^{2,\alpha}(M_\e)$.  This shows that the norm of $\linP_\e^{-1}:C^{0,\alpha}(M_\e)\to C^{2,\alpha}(M_\e)$ is
bounded independently of $\e$.
\end{proof}

We solve the Lichnerowicz equation by interpreting it as a fixed point problem. More precisely, consider the quadratic error term
$$
\Q_\e(\eta):=\ \Lie_\e(\psi_0+\eta)-\Lie_\e(\psi_0)-\linP_\e\eta
=\frac{1}{8}|\nu_\e|^2\Big((1+\eta)^{-7}-1+7\eta\Big)-\frac{3}{4}\Big((1+\eta)^5-1-5\eta\Big)
$$
and the corresponding map
$$\mathcal{G}_\e:\eta\mapsto -(\linP_\e)^{-1}\Big(\Lie_\e(\psi_0)+\Q_\e(\eta)\Big).$$
Note that, by Lemma \ref{N(1)estimate} and Theorem \ref{linLichn},
$$\mathcal{G}_\e:C^{2,\alpha}(M_\e)\to C^{2,\alpha}(M_\e) \text{\ \ and\ \ } \mathcal{G}_\e:C^{2,\alpha}_1(M_\e)\to C^{2,\alpha}_1(M_\e).$$
It is easy to see that the solutions $\psi_\e=\psi_0+\eta_\e$ of the Lichnerowicz equation correspond to the fixed points $\eta_\e$ of $\mathcal{G}_\e$. 
In what follows we argue that $\mathcal{G}_\e$ is a contraction mapping from a small ball in $C^{2,\alpha}(M_\e)$ to itself.

\begin{lemma}
For sufficiently large $C$ and sufficiently small $\e$,
the map $\mathcal{G}_\e$
is a contraction of the closed ball
of radius $C\sqrt{\e}$ around $0$ in $C^{2,\alpha}(M_\e)$.
\end{lemma}

\begin{proof}

Let $\eta_1, \eta_2\in C^{2,\alpha}(M_\e)$ be of norm $O(\sqrt{\e})$.
Assuming in addition that $|\eta_1|<1$ and $|\eta_2|<1$, using the bound on $|\nu_\e|_{g_\e}^2$ expressed in \eqref{nu:est}, and using the binomial expansion formulae, we find that for sufficiently small $\e>0$,
\begin{align*}
\|\Q_\e(\eta_2)-\Q_\e(\eta_1)\|_{0,\alpha}&=\left\| \frac{1}{8}|\nu_\e|^2 \Big[(1+\eta_1)^{-7}-(1+\eta_2)^{-7} +7(\eta_1-\eta_2)\Big]\right.\\
&\qquad \left.-\frac{3}{4}\Big[(1+\eta_1)^5- (1+\eta_2)^5 -5(\eta_1-\eta_2)\Big]\right\|_{0,\alpha}\\
&\le O(\sqrt{\e})\|\eta_2-\eta_1\|_{0,\alpha}
\le O(\sqrt{\e})\|\eta_2-\eta_1\|_{2,\alpha}.
\end{align*}

A similar calculation shows that if $\|\eta\|_{2,\alpha}=O(\sqrt{\e})$ then  $\|\Q_\e(\eta)\|_{0,\alpha}=O(\e)$.
As a consequence of  Lemma \ref{N(1)estimate}, functions $\eta$ with  $\|\eta\|_{2,\alpha}=O(\sqrt{\e})$ also satisfy
$$\|\Lie_\e(\psi_0)+\Q_\e(\eta)\|_{0,\alpha}=O(\sqrt{\e}).$$
Combining this with
Theorem \ref{linLichn} we have  that  there is a  sufficiently large constant $C>0$ such that for sufficiently small $\e>0$
$$\|\mathcal{G}_\e(\eta)\|_{2,\alpha}=\big\|(\linP_\e)^{-1}\big(\Lie_\e(1)+\Q_\e(\eta)\big)\big\|_{2,\alpha}\le C\sqrt{\e}.$$
Thus we have determined that  $\mathcal{G}_\e:\overline{B}_{C\sqrt{\e}}\to \overline{B}_{C\sqrt{\e}}$.

To see that this map is a contraction, we compute
$$\|\mathcal{G}_\e(\eta_1)-\mathcal{G}_\e(\eta_2)\|_{2,\alpha}=\big\|(\linP_\e)^{-1}\big(\Q_\e(\eta_1)-\Q_\e(\eta_2)\big)\big\|_{2,\alpha} =O(\sqrt{\e})\|\eta_1-\eta_2\|_{2,\alpha}.$$
It thus follows that if $\e>0$ is small enough, the map $\mathcal{G}_\e:\overline{B}_{C\sqrt{\e}}\to \overline{B}_{C\sqrt{\e}}$ is a contraction.
\end{proof}

We are now ready to state and prove our main result regarding solutions of  the Lichnerowicz equation for the parametrized sets of conformal data $(M_\e, g_\e, \nu_\e, \tau=3)$:

\begin{theorem}\label{thm:Lichnerowicz-main}
If $\e$ is sufficiently small, there exists a polyhomogeneous function $\psi_\e$ on $M_\e$ which has a $C^2$
extension to $\Mbar_\e$ that is equal to $1$ on $\partial\Mbar_\e$, and satisfies 
\begin{equation}\label{eq:Lich}
\Delta_{g_\e}\psi_\e-\frac{1}{8}R(g_\e) \psi_\e
+\frac{1}{8}|\nu_\e|_{g_\e}^2\psi_\e^{-7}-\frac{3}{4}\psi_\e^5=0.
\end{equation}
The function $\psi_\e$ is a small perturbation of the constant function $\psi_0\equiv 1$ in the sense that  
$$\|\psi_\e - \psi_0\|_{2,\alpha}=O(\sqrt{\e}) \text{\ \ as\ \ }\e\to 0.$$
\end{theorem}

\begin{proof}
By the Banach Fixed Point Theorem, 
the sequence 
$$\eta_{0,\e}:=0,\ \eta_{1,\e}:=\mathcal{G}_\e(\eta_{0,\e}),\ \dots \ , 
\eta_{n,\e}:=\mathcal{G}_\e(\eta_{n-1,\e}),\dots$$
converges in $\overline B_{C\sqrt{\e}}\subseteq C^{2,\alpha}(M_\e)$. 
Thus, there exists a function $\eta_\e$ on $M_\e$ such that
$\|\eta_\e\|_{2,\alpha}\le C\sqrt{\e}$ and such that
the function $\psi_\e:=\psi_0+\eta_\e$ solves the Lichnerowicz equation. 
To address the regularity of $\psi_\e$, 
note that $\eta_{n,\e}\in C^{2,\alpha}_1(M_\e)$ for all $n,\e$. Consequently, each $\eta_{n,\e}$
has a continuous extension  to $\Mbar_\e$ that
vanishes on $\partial\Mbar_\e$.
Because convergence in $C^{2,\alpha}(M_\e)$ implies uniform convergence, 
it follows that, for each fixed $\e$, the limit $\eta_\e:=\lim_{n\to \infty}\eta_{n,\e}$ 
also has a continuous extension to $\Mbar_\e$ and vanishes on $\partial\Mbar_\e$.
We now conclude that  $\psi_\e=\psi_0+\eta_\e$ approaches $1$ at the ideal boundary. Therefore, Corollary 7.4.2 of \cite{piot} applies and we see that $\psi_\e$ is polyhomogeneous. Inserting the asymptotic expansion for $\psi_\e$ into \eqref{eq:Lich}
and comparing like terms inductively,  we find that  the first log terms in 
$\psi_\e$ appear with $\rho_\e^3$. (These terms arise as a consequence of the indicial roots of the linearized Lichnerowicz operator.) 
It follows that  $\psi_\e$ has a $C^2$
extension to $\Mbar_\e$.
\end{proof}

With these solutions $\psi_\e$ to the Lichnerowicz equation in hand, we readily verify that the one-parameter family of initial data sets $(M_\e, \psi_\e^4 g_\e, \psi_\e^{-2} \nu_\e + \psi_\e^4 g_\e )$ satisfies the list of properties  outlined in \ref{thm:main}.
Hence we have constructed the desired asymptotic gluing of AH initial data satisfying the Einstein constraint equations.

\section{Conclusions}\label{sec:concl}

The gluing construction which we have discussed and verified here allows one to take a pair (or more) of CMC initial data sets for isolated systems with unique asymptotic regions--either asymptotically null data sets in asymptotically flat spacetimes, or data sets in asymptotically deSitter spacetimes--and glue them together in such a way that the spacetime which develops from this glued data has a single asymptotic region. In the case that the original data sets are asymptotically null, one may wonder how the Bondi mass \cite{Wald} for the glued data compares with the Bondi masses for the original data sets. We will study this issue in future work.

There are a number of ways in which the results proven here might be extended. It should be straightforward to be able to handle solutions of the Einstein-Maxwell or Einstein-fluid constraints, rather than the Einstein vacuum constraint equations. A more challenging generalization we plan to consider is to allow for initial data sets which do not have constant mean curvature. We have done this in earlier gluing work \cite{CIP} using localized deformations of the original data sets so that, in small neighborhoods of the gluing points, the mildly perturbed original initial data sets do have constant mean curvature. The work of Bartnik \cite{B} shows that this sort of deformation can always be done. A key first step in generalizing our results here to non CMC initial data sets is to generalize Bartnik's local CMC deformation results to neighborhoods of asymptotic points in AH initial data sets. This issue is under consideration. 

One further generalization of some interest is to attempt to carry out localized gluing at asymptotic points in AH initial data sets. To do this, it would likely be necessary to determine if the work of Chru\'sciel and Delay \cite{CD} generalizes so that it holds in asymptotic neighborhoods in AH initial data sets. While this may prove to be difficult, we do believe that we will be able to localize the gluing to the extent that in regions bounded away from the ideal boundary, the glued data is unchanged from the original data.

\end{document}